\documentclass[12pt,reqno,oneside]{amsart}
\usepackage{geometry}                
\usepackage[parfill]{parskip}    
\usepackage{graphicx}
\usepackage{amssymb}
\usepackage{epstopdf}

\usepackage{geometry}
\usepackage{graphicx}

\usepackage{booktabs} 
\usepackage{array} 
\usepackage{paralist} 
\usepackage{verbatim} 
\usepackage{subfig} 
\usepackage{tabularx}

\usepackage{amsmath,amsfonts,mathrsfs,amssymb,cite}
\usepackage[usenames]{color}

\newtheorem{thm}{Theorem}[section]

\newtheorem{lem}{Lemma}[section]

\newtheorem{rem}{Remark}[section]
\numberwithin{equation}{section}

\title[Recovering polyhedral obstacles]{ Recovering a polyhedral obstacle by a few backscattering measurements}

\author{Jingzhi Li}
\address{Faculty of Science, South University of Science and
Technology of China, 518055 Shenzhen, P.~R.~China.}
\email{li.jz@sustc.edu.cn}

\author{Hongyu Liu}
\address{Department of Mathematics, Hong Kong Baptist University,Kowloon Tong, Hong Kong SAR.\vspace*{-4mm}}
\address{\vspace*{-4mm}and}
\address{HKBU Institute of Research and Continuing Education, Virtual University Park, Shenzhen, P. R. China.}
\email{hongyu.liuip@gmail.com}

\begin{document}
\maketitle

\begin{abstract}

We propose an inverse scattering scheme of recovering a polyhedral obstacle in $\mathbb{R}^n$, $n=2,3$, by only a few high-frequency acoustic backscattering measurements. The obstacle could be sound-soft or sound-hard. It is shown that the modulus of the far-field pattern in the backscattering aperture possesses a certain local maximum behavior, from which one can determine the exterior normal directions of the front sides/faces. Then by using the phaseless backscattering data corresponding to a few incident plane waves with suitably chosen incident directions, one can determine the exterior unit normal vector of each side/face of the obstacle. After the determination of the exterior unit normals, the recovery is reduced to a finite-dimensional problem of determining a location point of the obstacle and the distance of each side/face away from the location point. For the latter reconstruction, we need make use of the far-field data with phases. Numerical experiments are also presented to illustrate the effectiveness of the proposed scheme.

\medskip

\noindent{\bf Keywords}. Inverse scattering, polyhedral obstacle, backscattering, phaseless

\medskip

\noindent{\bf Mathematics Subject Classification (2010)}:  78A46, 35Q60.

\end{abstract}

\section{Introduction}

This work concerns the inverse scattering problem of recovering an impenetrable obstacle by the corresponding acoustic wave detection. The problem has its physical origin in radar/sonar, geophical exploration, non-destructive testing and medical imaging (cf. \cite{CK,Isa2}). Let $D$ be a bounded Lipschitz domain in $\mathbb{R}^n$, $n=2,3$, such that $\mathbb{R}^n\backslash\overline{D}$ is connected. $D$ represents an impenetrable obstacle located in the space and it is assumed to be unknown/inaccessible. In order to identify $D$, one sends a time-harmonic detecting plane wave of the form
\begin{equation}\label{eq:planewave}
u^i(x)=e^{ix\cdot \xi},\quad \xi\in\mathbb{R}^n,\ \ \xi\neq 0,
\end{equation}
which is an entire solution to the Helmholtz equation $(-\Delta-|\xi|^2) u=0$ in $\mathbb{R}^n$. The presence of the obstacle $D$ interrupts the propagation of the plane wave, leading to the so-called scattered wave field $u^s$, which exists only in the exterior of the obstacle. The total wave field $u=u^i+u^s$ satisfies the following Helmholtz system
\begin{equation}\label{eq:Helm}
\begin{split}
& (-\Delta-|\xi|^2) u=0\quad\mbox{in\ \ $\mathbb{R}^n\backslash \overline{D}$},\quad \mathcal{B} u=0\quad\mbox{on\ \ $\partial D$},\\
&\lim_{|x|\rightarrow+\infty}|x|^{\frac{n-1}{2}}\left(\frac{\partial u^s}{\partial |x|}-i|\xi| u^s\right )=0.
\end{split}
\end{equation}
In \eqref{eq:Helm}, $\mathcal{B}u:=u$ or $\mathcal{B}u:=\partial u/\partial \nu$, with $\nu\in\mathbb{S}^{n-1}:=\{x\in\mathbb{R}^n; |x|=1\}$ denoting the exterior unit normal vector to $\partial D$, corresponding to that $D$ is sound-soft or sound-hard, respectively. In the following, we set $k=|\xi|\in\mathbb{R}_+$ and $d=\xi/|\xi|\in\mathbb{S}^{n-1}$, denoting the wavenumber and incident direction of the plane wave, respectively. The PDE system \eqref{eq:Helm} is well-understood with $u\in H^1_{loc}(\mathbb{R}^n\backslash\overline{D})$ possessing the following asymptotic expansion (cf. \cite{CK,McL})
\begin{equation}\label{eq:asym}
u(x)=e^{ix\cdot \xi}+\frac{e^{ik|x|}}{|x|^{\frac{n-1}{2}}} u^\infty(\frac{x}{|x|})+\mathcal{O}\left(\frac{1}{|x|^{\frac{n+1}{2}}}\right),\quad |x|\rightarrow+\infty,
\end{equation}
which holds uniformly in $\hat x:=x/|x|\in \mathbb{S}^{n-1}$, where $x\in\mathbb{R}^n$ and $x\neq 0$. $u^\infty(\hat x)$ is known as the far-field pattern and we shall write $u^\infty(\hat x; \xi, D)=u^\infty(\hat x; k, d, D)$ to specify its dependence on the observation direction $\hat x$, wavenumber $k$ and incident direction $d$, as well as the obstacle $D$. $u^\infty(\hat x)$ is real-analytic in $\hat x$, and hence if it is known on any open patch of $\mathbb{S}^{n-1}$, then it is known on the whole sphere by the analytic continuation; see \cite{CK}.

The inverse problem that we are concerned with in the present paper is to recover $D$ by knowledge of $u^\infty$, which is known to be nonlinear and ill-posed (cf. \cite{CK,Isa2}). It is noted that the inverse problem is formally posed with a fixed $\xi\in\mathbb{R}^n$ and all $\hat x\in\mathbb{S}^{n-1}$. Hence, there is a widespread belief that one can recover $D$ by using the far-field pattern corresponding to a single incident plane wave $e^{ix\cdot \xi}$, which is referred to as a single far-field measurement. However, this still remains to be a longstanding problem with very limited progresses in the literature. If the obstacle is of small size; roughly speaking, smaller than half of the detecting wavelength, the unique recovery result was established in \cite{CS}. If the obstacle is extremely ``rough" in the sense that its boundary is nowhere analytic, the unique recovery result was established in \cite{HonNakSin}. If the obstacle is of general polyhedral type, the corresponding study is almost exclusive \cite{AR,ElsYam,Liu1}. Recently, some qualitative numerical schemes of recovering the obstacles by a single far-field measurement were proposed in \cite{Ammari4,LLSS,LLZ}, where certain restrictive a priori assumptions have to be imposed on the obstacles. Another challenging issue in the study of inverse scattering problems is the recovery by phaseless data, say the modulus of the far-field pattern, $|u^\infty(\hat x)|$. The only recent result we are aware of is \cite{Kli}, where the unique determination of a scattering potential by the phaseless far-field measurements was established.

In this paper, we develop a novel scheme of recovering a polyhedral obstacle by using only a few high-frequency far-field measurements. The obstacle could be sound-soft or sound-hard. The crux is the observation that the modulus of the far-field pattern in the backscattering aperture possesses a certain local maximum behavior, from which one can determine the exterior normal directions of the front sides/faces. Then by using the modulus of the backscattering data corresponding a few incident plane waves with suitably chosen incident directions, one can determine the exterior unit normal vector of each side/face of the obstacle. After the determination of the exterior unit normals, the recovery is reduced to a finite-dimensional problem of determining a location point of the obstacle and the distance of each side/face away from the location point. For the latter reconstruction, the far-field data with phases would be used.
Our study is based on the high-frequency asymptotics, namely the Kirchhoff or the physical optics approximation. However, our numerical experiments show that the high-frequency requirement could be relaxed to a certain extent. Moreover, in order to simplify the discussion, we focus on convex polyhedral obstacles in the present study. However, through our theoretical arguments, it can be expected that the method developed would work for non-convex obstacles, but under certain geometrical constraints. We focus on developing the novel inverse scattering scheme for convex obstacles in the present study and leave the technical and tedious derivation of the geometrical conditions for non-convex obstacles for a forthcoming work.

The rest of the paper is organized as follows. In Section 2, we consider the physical optics approximation on the high-frequency scattering from an admissible polyhedral obstacle, and derive the local maximum behavior of the modulus of the corresponding far-field pattern. In Section 3, we present the recovery scheme in detail. Section 4 is devoted to numerical experiments to validate the applicability and effectiveness of the proposed method and the paper is concluded in Section 5 with some discussion.

\section{Physical optics approximation}

Throughout the present section, we let $k\in\mathbb{R}_+$ and $d\in\mathbb{S}^{n-1}$ be fixed.
Let $D$ be a convex polygon in $\mathbb{R}^2$ or a convex polyhedron in $\mathbb{R}^3$, such that
\begin{equation}\label{eq:boundary}
\partial D=\bigcup_{j=1}^m C_j,
\end{equation}
where each $C_j$ represents an open side/face of $\partial D$, and shall be referred to as a {\it cell} in what follows. In the sequel, $D$ is referred to as a polyhedral obstacle. Let $\nu(x)\in\mathbb{S}^{n-1}$, $x\in \partial D$ denote the exterior unit normal vector to $\partial D$, and we set
\begin{equation}\label{eq:normal}
\nu_j:=\nu(x)\ \ \mbox{when}\ \ x\in C_j,\ \ j=1,2,\ldots, m.
\end{equation}
Clearly, $\nu_j$ is a constant unit vector.

Define
\[
\partial D^{+}:=\{x\in\partial D;\ \nu(x)\cdot d\geq 0\}\quad\mbox{and}\quad \partial D^{-}:=\{x\in\partial D;\ \nu(x)\cdot d< 0\}
\]
to be, respectively, the back-face and front-face of $\partial D$ with respect to the incident direction $d$.

Let $h_0$, $h_1$ and $h_2$ be fixed a priori positive constants. It is further assumed that
\begin{align}
(i)&~~k\cdot \mbox{diam}(D)\gg 1;\hspace{5cm} \empty\label{eq:cond1} \\
(ii)&~~\displaystyle{\min_{1\leq j\leq m} \mbox{diam}_{\mathbb{R}^{n-1}}(C_j)\geq h_0}; \label{eq:cond2}\\
(iii)&~~h_1\leq \displaystyle{\min_{1\leq j, j'\leq m,\ j\neq j'} \angle(\nu_j(y), \nu_{j'}(y)) }\leq h_2\ \ \mbox{for}\ \ y\in \partial D. \label{eq:cond3}
\end{align}
Roughly speaking, \eqref{eq:cond3} implies that the obstacle should not be very ``round" or ``sharp", and a generic condition which can guarantee this assumption is that the angle between any two adjacent cells is bounded below and above by certain constants (depending on the obstacle).  Assumption~(i) means that we are considering the scattering in the high-frequency regime. If a polyhedral obstacle $D$ satisfies the above three assumptions, then it is called an {\it admissible} obstacle with respect to the incident plane wave $e^{ikx\cdot d}$.

Denote
\[
\mathbb{S}^{n-1}_+:=\{\hat x\in\mathbb{S}^{n-1}; \ \hat x\cdot d\geq 0\}\quad\mbox{and}\quad \mathbb{S}^{n-1}_-:=\{\hat x\in\mathbb{S}^{n-1}; \ \hat x\cdot d<0\}
\]
the forward-scattering and backscattering apertures, respectively.

Let $C_j\subset\partial D^-$ be a front-cell of $\partial D$, and $\nu_j\in \mathbb{S}^{n-1}_-$ denote its unit normal vector pointing to the exterior of $D$.  Define $\hat x_j\in\mathbb{S}^{n-1}$ satisfying
\[
(d-\hat x_j)\parallel \nu_j
\]
to be the critical observation direction with respect to $d$ and $\nu_j$; see Fig.~1 for a 2D illustration. We note that one clearly has $\hat x_j\in\mathbb{S}_-^{n-1}$. It is directly calculated that the critical direction is given by
\begin{equation}\label{eq:cd}
\hat{x}_j=d-2(d\cdot \nu_j)\nu_j.
\end{equation}
On the other hand, for the subsequent use, we note that by using \eqref{eq:cd} and the fact that $d\cdot\nu_j<0$, one has by innerly producting both sides of \eqref{eq:cd} with $d$
\begin{equation}\label{eq:cd2}
d\cdot\nu_j=-\sqrt{\frac{1-\hat{x}_j\cdot d}{2}}.
\end{equation}
Hence, by combining \eqref{eq:cd} and \eqref{eq:cd2}, one further has
\begin{equation}\label{eq:cd3}
\nu_j=\frac{\hat x_j-d}{\sqrt{2(1-\hat{x}_j\cdot d)}}.
\end{equation}


\begin{figure}
\centering
\includegraphics[width=0.5\textwidth]{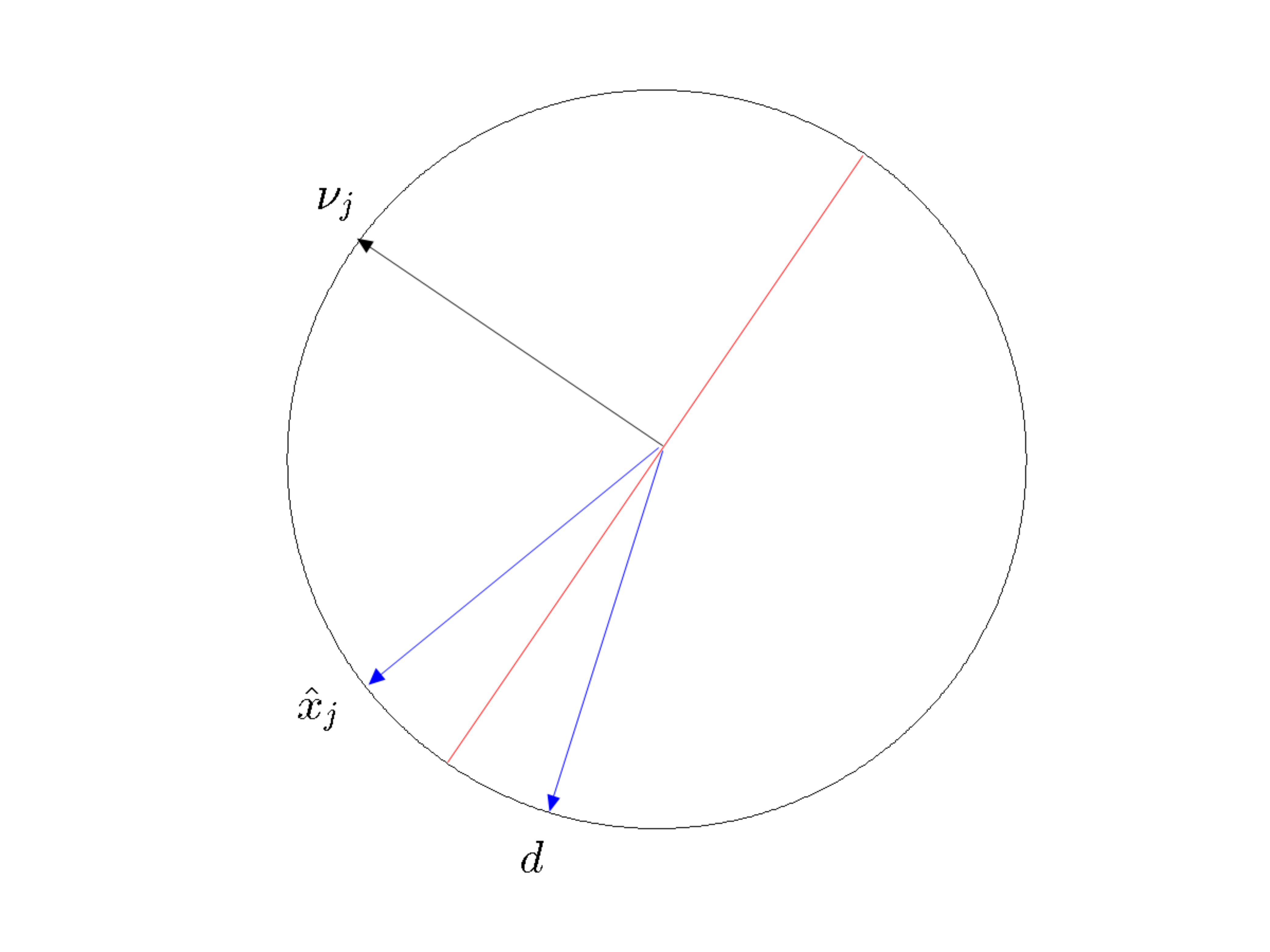}

\caption{\label{fig:1} 2D illustration of the relation between the incident direction $d$, the exterior unit normal vector $\nu_j$ and the critical observation direction $\hat{x}_j$. }
\end{figure}

\begin{figure}
\centering

\hfill{}\includegraphics[width=0.48\textwidth]{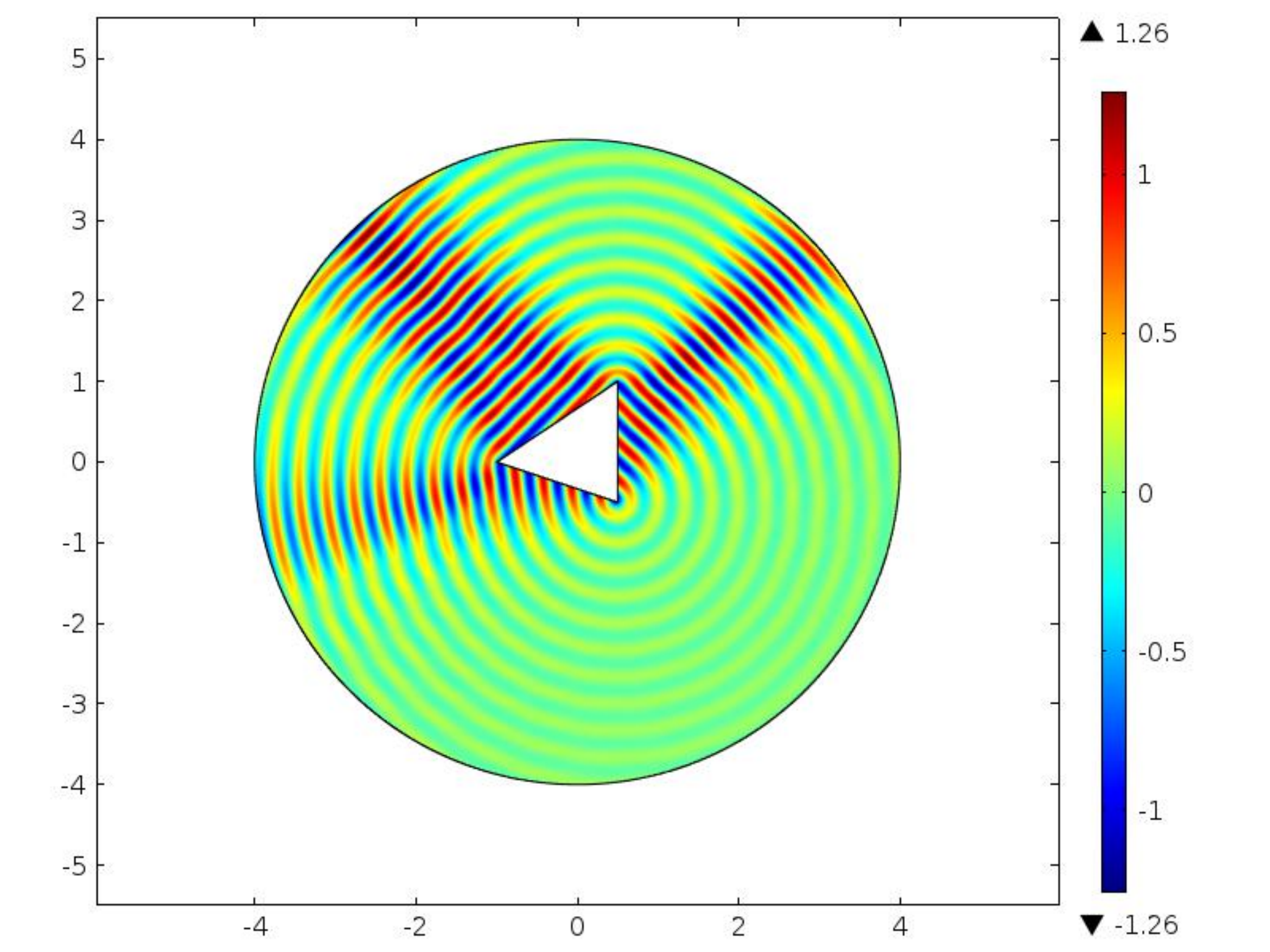}\hfill{} \includegraphics[width=0.48\textwidth]{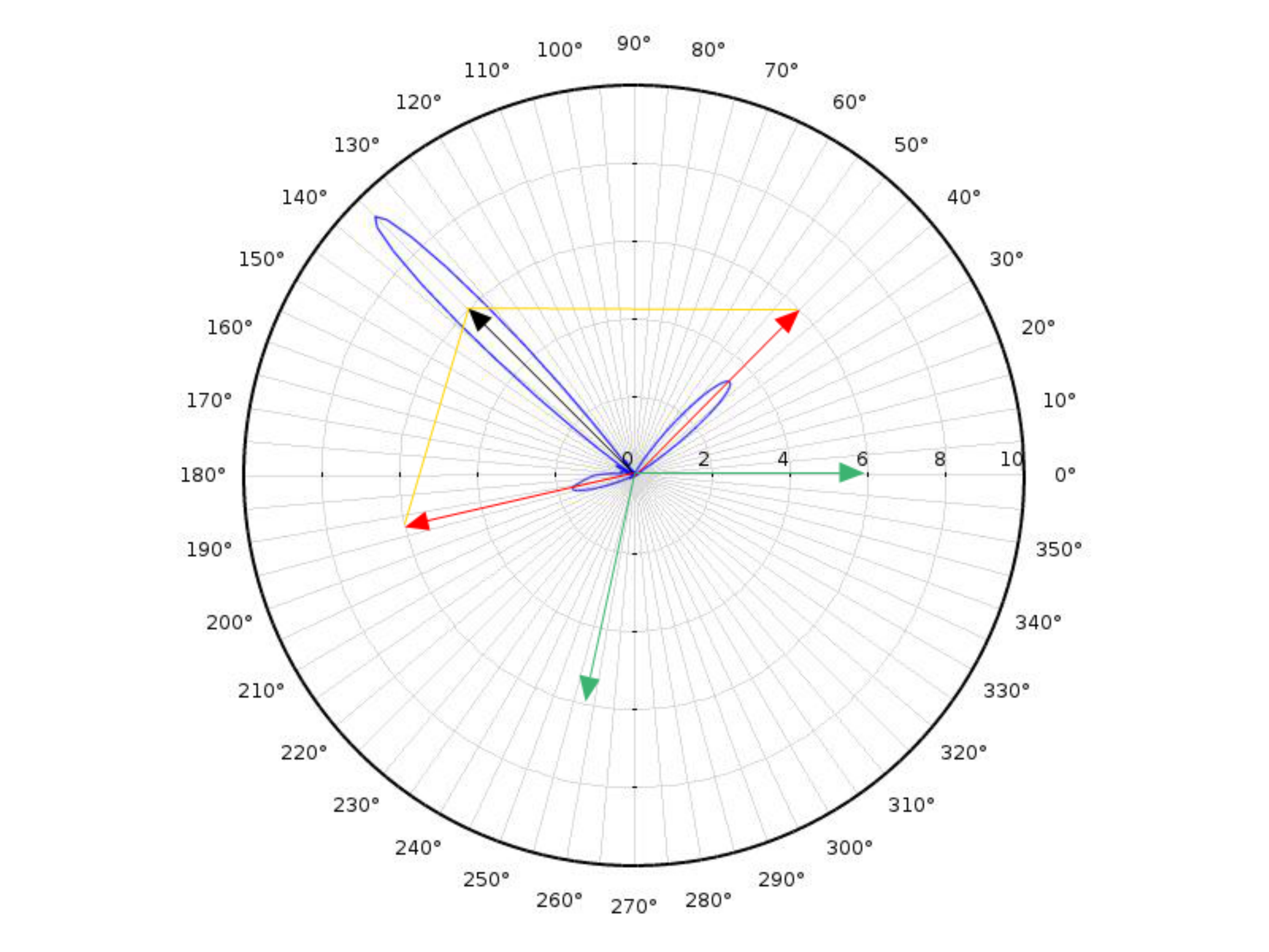}\hfill{}

\hfill{}(a)\hfill{}\hfill{}(b)\hfill{}

\caption{\label{fig:2} (a) Plot of the real part of the scattered wave $\Re(u^s(x))$ in the vicinity of a sound-soft triangular obstacle, and (b) polar graph of the square power of the phaseless far-field data $|u^\infty(\hat x)|^2$ in polar coordinates corresponding to the triangle due to an incident plane wave $e^{ikx\cdot d}$ with $d=(-\sqrt{2}/2, \sqrt{2}/2)$. The black arrow indicates both the incident and the forward-scattering directions, while the two red arrows represents the two critical observation directions. The green arrows indicates the exterior directions normal to the sides of the triangular obstacle.   }
\end{figure}

Next, we present the major result motivating the recovery scheme that we are going to develop in the next section. Before that, we give a numerical example by plotting the real part of the scattered wave field and the square power of the modulus of the associated far-field pattern corresponding to a sound-soft triangle due to an incident plane wave impinging from southeast to northwest; see Fig.~\ref{fig:2} for the illustration.

Clearly, one can see a certain local maximum behavior from Fig.~\ref{fig:2}(b). The observation direction (black arrow) associated with the largest magnitude of phaseless data points to the incident/forward-scattering direction, which gives no information of the obstacle. While the other two critical observation directions (red arrows) associated with the second and third largest magnitude of phaseless data provide profound information of unknown obstacle, which signify the major reflection angles due to the physical optics approximation in the high frequency scattering. The exterior unit normal
direction can thus be determined by connecting the arrow heads of incident and critical observation directions by yellow lines. Those unit vector (green arrows) parallel to the yellow lines are the desired normal directions of those sides of the triangular obstacle in the backscattering aperture as shown in Fig.~\ref{fig:2}(b). We shall present a rigorous mathematical justification for this local maximum behavior, which also forms the basis of our subsequent recovery scheme.

We first give some discussion on the Kirchhoff or the physical optics approximation (cf. \cite{CWL,CK,HLM,LaxPhi,Maj,MelTay}), which plays a key role in our recovery scheme. Considering the present study, it says that for the scattering from a convex polyhedral obstacle due to a high-frequency plane wave, the wave field near the boundary of the obstacle is composed of two parts: the contribution of the incident and reflected waves where they are present, and the contribution of the diffraction due to the corners and/or the edges of the obstacle. The Kirchhoff or the physical optics approximation takes the first contribution as the total wave field near the boundary of the obstacle. However, we would like to remark the study on rigorously justifying such approximation is still not fully understood and we refer to \cite{CWL} for an excellent account on the existing progress in the literature. For our study, we assume that the physical optics approximation holds true. 

Throughout the rest of the paper, we let a cell $C$ be parameterized as
\begin{equation}\label{eq:cell1}
\langle x, \nu \rangle=l,\quad x\in C,
\end{equation}
where $\nu\in\mathbb{S}^{n-1}$ is the unit normal vector to $C$ pointing to the exterior of the obstacle, and $l\geq 0$ denotes the distance between the origin and the line/plane containing $C$. Let $C^0$ denote the affine cell of $C$ defined by
\begin{equation}\label{eq:cell2}
\langle x, \nu\rangle=0,\quad x\in C^0.
\end{equation}
In the sequel, we let $\Pi_{C^0}$ denote the Euclidean reflection with respect to $C^0$. Now, we consider the scattering locally near a boundary cell, say $C_j$ due to a plane wave $e^{ikx\cdot d}$. Let
\begin{equation}\label{eq:k1}
v_j(x):=e^{ik(x-x_0)\cdot (\Pi_{C_j^0} d)}\cdot e^{ikx_0\cdot d},\quad 1\leq j\leq m,
\end{equation}
where $x_0\in C_j$ is a fixed point, and $x\in \mathbb{R}^n\backslash\overline{D}$.
It is easily verified that $v_j(x)$ satisfies $(\Delta+k^2)v_j=0$. Moreover, one can further verify that
\begin{equation}\label{eq:k2}
u^i(x)-v_j(x)=0\quad\mbox{and}\quad\frac{\partial (u^i+v_j)}{\partial\nu}(x)=0\quad\mbox{on}\ C_j.
\end{equation}
$v_j(x)$ is the reflected wave field of $u^i(x)$ with respect to the cell $C_j$. Using the physical optics approximation, we have

\begin{lem}\label{thm:ka}
Let $D$ be an admissible polyhedral obstacle with respect to the plane wave field $u^i=e^{ikx\cdot d}$, as described in \eqref{eq:boundary}--\eqref{eq:cond3}. Using the physical optics approximation, one has
\begin{equation}\label{eq:sa}
\frac{\partial u}{\partial \nu}(x)\approx\begin{cases}
& 2 \frac{\partial u^i}{\partial\nu}(x),\quad x\in {C}_j\subset\partial D^-, \ \ 1\leq j\leq m,\\
&\ \ \ 0,\qquad\ \, x\in C_{j'}\subset\partial D^+,\ \ 1\leq j'\leq m,
\end{cases}
\end{equation}
if $D$ is sound-soft; and
\begin{equation}\label{eq:ha}
u(x)\approx
\begin{cases}
& 2 u^i(x), \quad x\in {C}_j\subset\partial D^-, \ \ 1\leq j\leq m,\\
& \ \ 0,\qquad\ \, x\in C_{j'}\subset\partial D^+,\ \ 1\leq j'\leq m,
\end{cases}
\end{equation}
if $D$ is sound-hard.
\end{lem}
\begin{proof}
According to our earlier discussion on the physical optics approximation, one takes
\begin{equation}\label{eq:sss1}
\frac{\partial u}{\partial \nu_j}(x)\approx \frac{\partial (u-v^j)}{\partial \nu_j}(x),\quad x\in C_j\subset\partial D^-,
\end{equation}
if $D$ is sound-soft. Then by using \eqref{eq:k1}, one directly verifies \eqref{eq:sa} for $x\in\partial D^-$, whereas if $C_{j'}\subset\partial D^+$, it is not illuminated, where one then takes $\partial{u}/{\partial\nu_{j'}}\approx 0$. If $D$ is sound-hard, then one takes
\begin{equation}\label{eq:sss2}
u(x)\approx u(x)+v^j(x),\quad x\in C_j\subset\partial D^-,
\end{equation}
which readily verifies \eqref{eq:ha}. 
\end{proof}

It is emphasized again that Theorem~\ref{thm:ka} is mainly based on physical observation, and its rigorous justification still largely remains open in the literature. Nevertheless, our subsequent study on the inverse scattering problem, along with the corresponding numerical experiments, validates such approximation as well.


In the sequel, we let
\[
\Phi(x, y)=\frac{e^{ik|x-y|}}{4\pi |x-y|},\quad n=3; \quad \frac{i}{4} H_0^{(1)}(k |x-y|),\quad n=2; \quad x\neq y,
\]
where $H_0^{(1)}$ denotes the zeroth-order Hankel function of the first kind. $\Phi$ is the fundamental solution to $-\Delta-k^2$. The following lemma shall be needed and its proof can be found in \cite{CK}.
\begin{lem}\label{lem:1}
For the scattering of a plane wave field $u^i$ in \eqref{eq:planewave} from an obstacle $D$, we have
\begin{equation}\label{eq:farf1}
u(x; D, u^i)=u^i(x)+\int_{\partial D}\left\{ \frac{\partial \Phi(x,y)}{\partial \nu(y)} u(y)-\Phi(x,y)\frac{\partial u}{\partial\nu}(y)\right\}\ ds(y),\quad x\in\mathbb{R}^n\backslash\overline{D},
\end{equation}
and
\begin{equation}\label{eq:farf2}
u^\infty(\hat x; D, u^i)=\gamma(n,k)\left[ \int_{\partial D}\left\{\frac{\partial e^{-ik \hat x\cdot y}}{\partial\nu(y)} u(y)-e^{-ik\hat x\cdot y}\frac{\partial u}{\partial\nu}(y)\right\}\ ds(y) \right],
\end{equation}
where the dimensional parameter $\gamma$ is given by
\begin{equation}\label{eq:gamma}
\gamma(n,k)=\frac{1}{4\pi}\quad \mbox{when}\ n=3;\quad \frac{e^{i\frac\pi 4}}{\sqrt{8\pi k}}\quad\mbox{when}\ n=2.
\end{equation}
\end{lem}

\begin{lem}\label{lem:2}
Using the physical optics approximation in Lemma~\ref{thm:ka}, for the scattering of a plane wave $u^i$ in \eqref{eq:planewave} from an admissible polyhedral obstacle $D$, one has that if $D$ is sound-soft
\begin{equation}\label{eq:k3}
u^\infty(\hat x)\approx-2\gamma(n,k)\int_{\partial D^-}\frac{\partial e^{iky\cdot d}}{\partial\nu(y)} e^{-ik\hat x\cdot y}\ ds(y),
\end{equation}
whereas if $D$ is sound-hard
\begin{equation}\label{eq:k4}
u^\infty(\hat x)\approx 2\gamma(n,k)\int_{\partial D^-}\frac{\partial e^{-ik\hat x\cdot y}}{\partial \nu(y)} e^{iky\cdot d}\ ds(y).
\end{equation}
\end{lem}

\begin{proof}
This is a straightforward consequence of \eqref{eq:farf2} in Lemma~\ref{lem:1} and, \eqref{eq:sa} and \eqref{eq:ha} in Lemma~\ref{thm:ka}.
\end{proof}

By using Lemma~\ref{lem:2}, we are now in a position to present the local maximum behavior of $|u_\infty(\hat x; D, k, d)|$.
\begin{thm}\label{thm:1}
Let $D$ be an admissible sound-soft or sound-hard polyhedral obstacle with respect to the incident plane wave $e^{ikx\cdot d}$ as described earlier. Suppose that $C_j\subset\partial D^-$ is a front cell of the obstacle, and $\nu_j$ is the unit normal vector to $C_j$ pointing to the exterior of $D$, $1\leq j\leq m$. Let $\hat x_j\in\mathbb{S}^{n-1}$ be the critical observation direction with respect to $d$ and $\nu_j$. Under the physical optics approximation, $\hat{x}_j$ is a local maximum point of $|u^\infty(\hat x; D, e^{ikx\cdot d})|$.
\end{thm}

\begin{proof}
We first consider the case that $D$ is an admissible sound-soft polyhedral obstacle. By \eqref{eq:k3} in Lemma~\ref{lem:2}, we have
\begin{equation}\label{eq:k5}
\begin{split}
& u^\infty(\hat x; D, k, d)\approx -2\gamma(n,k)\int_{\partial D^-}\frac{\partial e^{iky\cdot d}}{\partial\nu(y)} e^{-ik\hat x\cdot y}\ ds(y)\\
=& -2\gamma(n,k)\int_{\partial D^-} ik \nu(y)\cdot d e^{iky\cdot (d-\hat x)}\ ds(y)\\
=& -2\gamma(n,k)\cdot ik\left[\int_{C_j}\nu_j\cdot d\ e^{ik y\cdot (d-\hat x)}\ ds(y)+ \int_{\partial D^-\backslash C_j}\nu(y)\cdot d\ e^{ik y\cdot (d-\hat x)}\ ds(y) \right]\\
=& \widetilde{\gamma}(n,k)\cdot [I_1(\hat x)+I_2(\hat x)],
\end{split}
\end{equation}
where $\widetilde{\gamma}(n,k):=-2\gamma(n,k)\cdot ik$ and
\[
I_1(\hat x):=\int_{C_j}\nu_j\cdot d\ e^{ik y\cdot (d-\hat x)}\ ds(y),\quad I_2(\hat x):= \int_{\partial D^- \backslash C_j}\nu(y)\cdot d\ e^{ik y\cdot (d-\hat x)}\ ds(y).
\]

We next analyze the behavior of $I_\alpha (\hat x)$, $\alpha=1,2$, in a small neighborhood of $\hat x_j$ on $\mathbb{S}^{n-1}$, say $\Gamma_j$, $1\leq j\leq m$. Since
\begin{equation}\label{eq:kk1}
(d-\hat x_j)\parallel \nu_j,
\end{equation}
by assumption~\eqref{eq:cond3}, we see that there exists a non-asymptotic constant $\epsilon_0\in\mathbb{R}_+$ such that
\begin{equation}\label{eq:kk2}
|\tau(y)\cdot (d-\hat x_j)|\geq \epsilon_0,\ \ y\in\partial D^-\backslash C_j,
\end{equation}
where $\tau(y)$ represents a unit tangent vector on $\partial D$. Hence, by \eqref{eq:kk2} and \eqref{eq:cond2}, one has by direct calculations that
\begin{equation}\label{eq:kk3}
|I_2(\hat x)|\sim \frac{1}{k^{n-1} h_0}\ll 1,\quad \hat x\in\Gamma_j.
\end{equation}
On the other hand, for $\hat x\in\Gamma_j$, we have
\begin{equation}\label{eq:kk4}
\begin{split}
I_1(\hat x)=& \int_{C_j}\nu_j\cdot d\ e^{ik y\cdot (d-\hat x)}\ ds(y)\\
=& e^{ik y_0\cdot (d-\hat x)}\int_{C_j^0} \nu_j\cdot d\ e^{ik y\cdot (d-\hat x)}\ ds(y),
\end{split}
\end{equation}
where $y_0$ is any fixed point on $C_j$.
Since
\[
y\cdot (d-\hat x_j)=0\ \ \mbox{for}\ \ y\in C_j^0,
\]
one clearly sees that $|I_1(\hat x)|$ achieves its local maximum value at $\hat x_j$, which in combination with \eqref{eq:kk3} completes the proof of the theorem for the sound-soft case. 

The sound-hard case can be shown by following a similar argument. By using \eqref{eq:k4}, one has
\begin{equation}\label{eq:kk5}
\begin{split}
u^\infty(\hat x; D, k, d)\approx & -2\gamma(n,k)\cdot ik\int_{\partial D^-} \nu(y)\cdot \hat x\ e^{ik y\cdot(d-\hat x)}\ ds(y)\\
=& \widetilde{\gamma}(n,k)\cdot [J_1(\hat x)+J_2(\hat x)],
\end{split}
\end{equation}
where
\begin{equation}\label{eq:kk6}
J_1(\hat x)=\int_{C_j}\ \nu_j\cdot \hat x \ e^{ik y\cdot (d-\hat x)}\ ds(y),\quad J_2(\hat x)=\int_{\partial D^- \backslash C_j}\ \nu(y)\cdot \hat x \ e^{ik y\cdot (d-\hat x)}\ ds(y).
\end{equation}
We consider the behavior of $J_\alpha(\hat x)$, $\alpha=1,2$, in a small neighborhood of $\hat x_j$ on $\mathbb{S}^{n-1}$, say $\Sigma_j$, $1\leq j\leq m$. Following a similar argument to \eqref{eq:kk4}, together with the fact \eqref{eq:kk1}, one can see that $|J_1(\hat x)|$ achieves its local maximum value at $\hat x_j$. On the other hand, by a similar argument in deriving \eqref{eq:kk3}, one has that 
\begin{equation}\label{eq:kkk0}
|J_2(\hat x)|\ll 1\quad\mbox{for}\ \ \hat x\in\Sigma_j,\ \ 1\leq j\leq m. 
\end{equation}
Hence, $\hat x_j$ is local maximum point of $|u^\infty(\hat x; D, e^{ikx\cdot d})|$.

The proof is complete. 

\end{proof}

\begin{rem}
In Theorem~\ref{thm:1}, we only consider the local maximum behavior of $|u^\infty(\hat x)|$ in the backscattering aperture. In fact, by \eqref{eq:k5}, one clearly sees that $|u^\infty(\hat x)|$ achieves its (global) maximum value at $\hat x=d$. However, this maximum behavior in the forward-scattering aperture gives us no information about the obstacle, which is demonstrated and verified in Fig.~\ref{fig:2}.
\end{rem}

\begin{rem}
It is noted from \eqref{eq:kk5} that for a sound-hard obstacle if $\hat x\perp \nu_j$, then both $J_1$ and $J_2$ vanish, and hence such an $\hat x$ should be a local minimum point of $|u^\infty(\hat x; D, k, d)|$. However, it cannot be guaranteed that such an $\hat x$ belongs to the backscattering aperture. Nevertheless, it could be used in combination with the critical observation direction as a double indicator.
\end{rem}

\section{Recovery scheme}

Based on our study  in the previous section, we shall present a recovery scheme of identifying an admissible polyhedral obstacle $D$. To that end, we let $k$ be a sufficiently large wavenumber such that assumption~\ref{eq:cond1} is fulfilled. Then, let $d_\alpha$, $\alpha=1,2,\ldots, p$ be a few properly chosen incident directions. The basic requirement is that the union of the front-faces with respect to $d_\alpha$, $\alpha=1,2,\ldots, p$ should cover the whole boundary $\partial D$. Furthermore, we shall assume that $D$ is an admissible obstacle with respect to each $d_\alpha$. As remarked earlier, this assumption is generically satisfied by an overall not very ``round" obstacle.

First of all, clearly, by Theorem~\ref{thm:1} and the discussion following its proof, one can recover a family of exterior unit normal vectors. However, it must be emphasized that multiple groups of normal directions are actually determined corresponding to detecting waves with different incident directions. Within each group, those normal directions point roughly to the same direction. Up to this stage, we choose the normal direction, which is associated with the critical observation angle with the largest magnitude of phaseless data within the group, to be an \emph{effective} exterior unit normal direction. Following this principle, we can recover all the exterior unit normal vector $\nu_j$ to each of the cell $C_j$ to $\partial D$, $j=1,2,\ldots, m$.
In turn, we also know the number of cells of the obstacle.

Next, we proceed to identify each cell, and we let $x_0\in D$ be a fixed location point. Let each $C_j$ be parameterized as follows
\begin{equation}\label{eq:pp11}
\langle x-x_0, \nu_j\rangle=l_j,\quad x\in C_j,
\end{equation}
where $l_j$ denotes the distance between the origin and the line/plane containing the cell $C_j-x_0:=\{x-x_0; x\in C_j\}$.
We note that $\nu_j$ and $l_j$, $j=1,2,\ldots, m$, and $x_0$ uniquely determine the convex obstacle $D$. It is highlighted that
the location point $x_0$ can be initially guessed by Scheme I in \cite{LLZ} using single-shot measurement data at low frequency.
Hence, after the determination of $x_0$ and $\nu_j$, $j=1,2,\ldots, m$, the corresponding inverse problem reduces to a finite dimensional problem of finding $l_j$, $j=1,2,\ldots, m$. But we have richer data set $u^\infty(\hat x; k, d_\alpha)$ to that purpose.

Now, let us take the sound-soft obstacle as an illustration to derive the rest of the recovery scheme. Suppose $\hat{x}^\alpha_j$ is a critical observation direction corresponding to $C_j$ with respect to a certain $d_{\alpha}$, $1\leq\alpha\leq p$.
Then, by \eqref{eq:k5}, we have
\begin{equation}\label{eq:pp2}
u^\infty(\hat{x}^\alpha_j; k, d_{\alpha})\approx \widetilde{\gamma}(n, k) \int_{C_j}\nu_j\cdot d_{\alpha}\ e^{ik y\cdot (d_{\alpha}-\hat{x}^\alpha_j)}\ ds(y)\,. 
\end{equation}
Eq.~\eqref{eq:pp2} is clearly a finite dimensional nonlinear problem with respect to $l_j$'s, which are hidden in the implicit functions defining $C_j$'s.

By supplementing $m$ particularly chosen normal directions and their respective  critical observation angles, one can arrive at a nonlinear system of $m$ equations with $m$ unknown $l_j$, $j=1,2,\ldots, m$. The cells can be determined all at once by solving the finite dimensional problem.
It is emphasized that the chosen $m$ nonlinear equations \eqref{eq:pp2} hold only approximately, thus we convert it into a nonlinear least-squares minimization problem to yield a more stable and accurate identification.

Now we are in a position to present our main reconstruction algorithm of polygonal obstacles which is sketched in Fig.~\ref{fig:11}

\medskip

\noindent {\underline {\bf Recovery Scheme}.}

\medskip

Step 1. Employ Scheme I proposed in \cite{LLZ} to locate the position $x_0$ of the polygonal scatterer using an incident detecting wave of low frequency.

Step 2. Determine the $m$ effective exterior normal vectors $\nu_j$, $j=1,2,\ldots,m$ and their associated critical observation angles from multiple groups of candidate normal vectors associated with all the critical observation angles.

Step 3. Given the distances $l_j$'s from $x_0$ along the direction $\nu_j$, $j=1,2,\ldots,m$, we can locate all the perpendicular points $P_j$'s.

Step 4. Extend the line/plane at $P_j$'s perpendicular to the respective $\nu_j$ and denote those crossing points to be vertices $K_j$'s. The line/plane of the polygonal scatterer can be determined by $C_j:=\overline{K_j K_{j'}}$ where $j'=j+1$ if $j<m$ and $j'=1$ if $j=m$.

Step 5. Select the first $m$ critical observation angles $\hat{z}_\beta$, $\beta:=\beta(\alpha,j)=1,2,\ldots, m$ from all the $\hat{x}^\alpha_j$, $\alpha=1,2,\ldots,p$, $j=1,2,\ldots,m$ according to the decreasing order of magnitude of local backscattering maxima and convert the corresponding $m$ nonlinear equations associated with \eqref{eq:pp2}  with $m$ unknowns $l_j$'s into a least-squares minimization problem. It is noted that the integral to the right hand side of \eqref{eq:pp2} is approximated by the trapezoidal quadrature rule with sufficiently fine step size.
\begin{eqnarray*}
  \min   F(t)  &=&  \sum_{\beta=1}^m \left|u^\infty(\hat{z}_\beta; k, d_{\alpha})- \widetilde{\gamma}(n, k) \int_{C_j}\nu_j\cdot d_{\alpha}\ e^{ik y\cdot (d_{\alpha}-\hat{z}_\beta)}\ ds(y)\right|^2  \label{eq:costfun}\\
  t &=&(l_1,l_2,\ldots,l_m)^T. \nonumber
\end{eqnarray*}

Step 6. The minimum of the  multivariable nonlinear least-squares minimization problem is obtained by employing a derivative-free trust region method via a local quadratic surrogate model-based search algorithm. Interested readers may refer to \cite{CSV} and the references therein.



\begin{figure}
\centering
\includegraphics[width=0.6\textwidth]{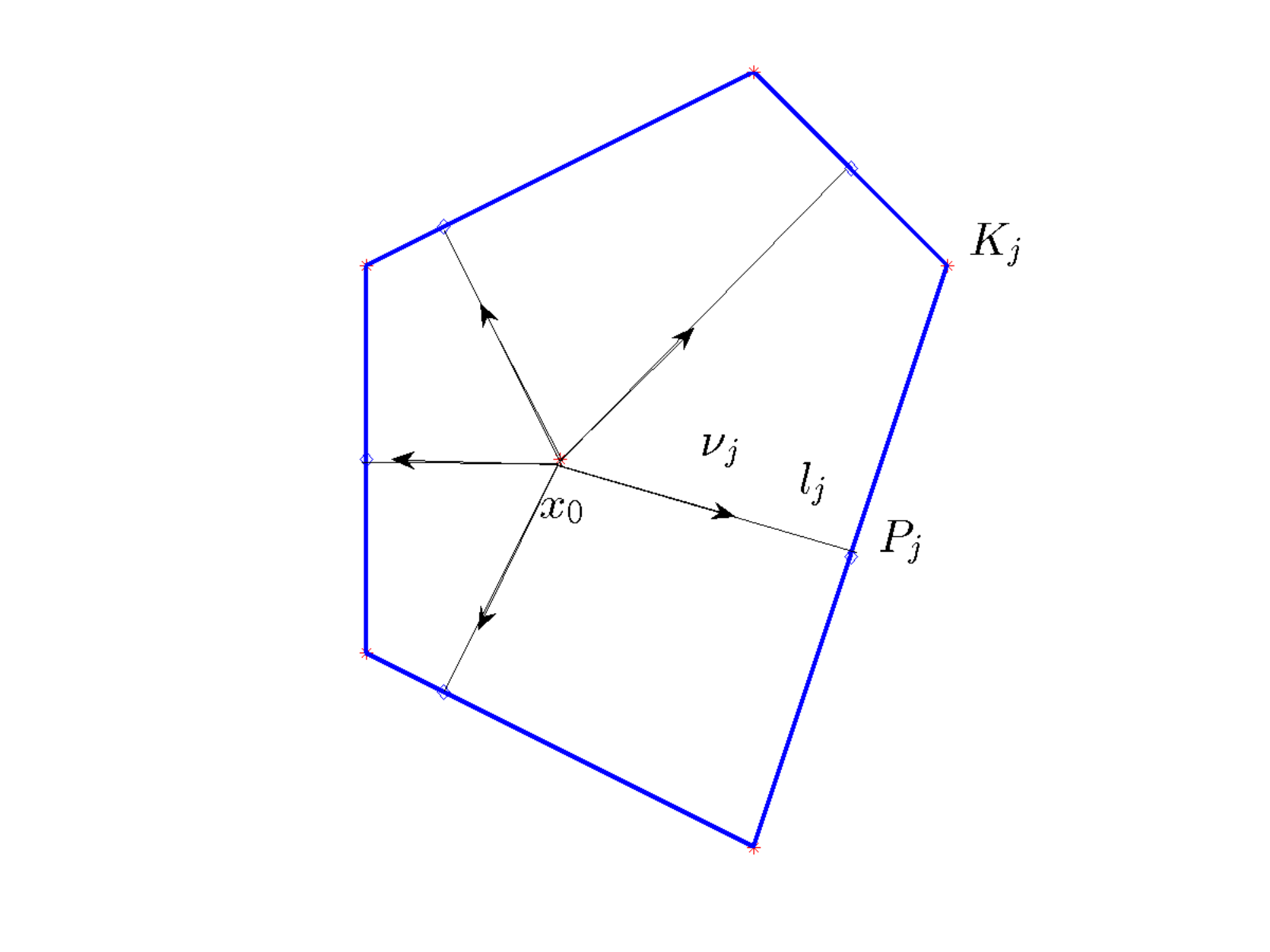}
\caption{\label{fig:11} Sketch of the recovery scheme. }
\end{figure}

\section{Numerical experiments and discussions}

In this section, we present some numerical tests to verify the applicability and effectiveness of the proposed recovery scheme in two dimensions.
In the sequel,  the forward equation (\ref{eq:Helm}) is first solved
by using the quadratic finite element discretization  on a truncated circular  domain enclosed by a PML layer.
The forward solver is iterated on a sequence of successively refined meshes till the relative error of two successive finite element solutions between the two adjacent meshes is below $0.1\%$. Then the scattered data are transformed into the far-field data by employing the Kirchhoff integral formula on a closed circle (2D)  enclosing the scatterer.

First of all, let's fix the parameter settings. For the positioning purpose , we apply  Scheme I (see \cite{LLZ}) with $k=1$ to detect the initial displacement guess $x_0$ of the unknown polygonal scatterer. A few detecting waves with the incident directions chosen among the set $\{ d_j=(\cos(j\pi/4),\sin(j\pi/4))\}$, $j=1,2,\ldots,8$, are sent off for locating the sides of the polygonal scatterer.  Then the far-field data
are measured and collected at 360 equidistant observation angles along the unit circle. The far-field data generated on the unit circle are then subjected
pointwise to certain uniform random noise. The uniform random noise
in magnitude as well as in direction is added according to the
following formula,
\begin{equation}
u^\infty=u^\infty+\delta \,r_{1}|u^\infty|\exp(i\pi \,r_{2})\,,
\end{equation}
where  $r_{1}$ and $r_{2}$ are two uniform random
numbers, both ranging from -1 to 1, and $\delta$ represents the
noise level.

For polygonal obstacles, we shall test sound-soft and sound-hard, noise-free and noisy cases, respectively. In the noise-free case, it is well-known that the far-field data is analytic and thus very smooth. As a consequence, the local maximum behavior of phaseless far-field data is clear from its polar graph as shown in Fig.~\ref{fig:2} when there exists no noise. While in the noisy case, we always add to the exact far-field data a uniform noise of $5\%$ and use it as the noisy measurement data, which is inevitable from the practical point of view. But the local maximum behavior of phaseless far-field data might be corrupted by the ups and downs of random noise which cause fictitious and/or more local mamima than usual. The cure out of the dilemma is to add a preprocessing step to filter the raw noisy data. In our tests, a fourier filtering stage is applied to the noisy far-field data in advance.  The filtered measurement data are thus used as in the recovery scheme, which gives us better reconstructions than using raw data by our experience.


\begin{figure}[h]
\centering
\hfill{}\includegraphics[width=0.43\textwidth]{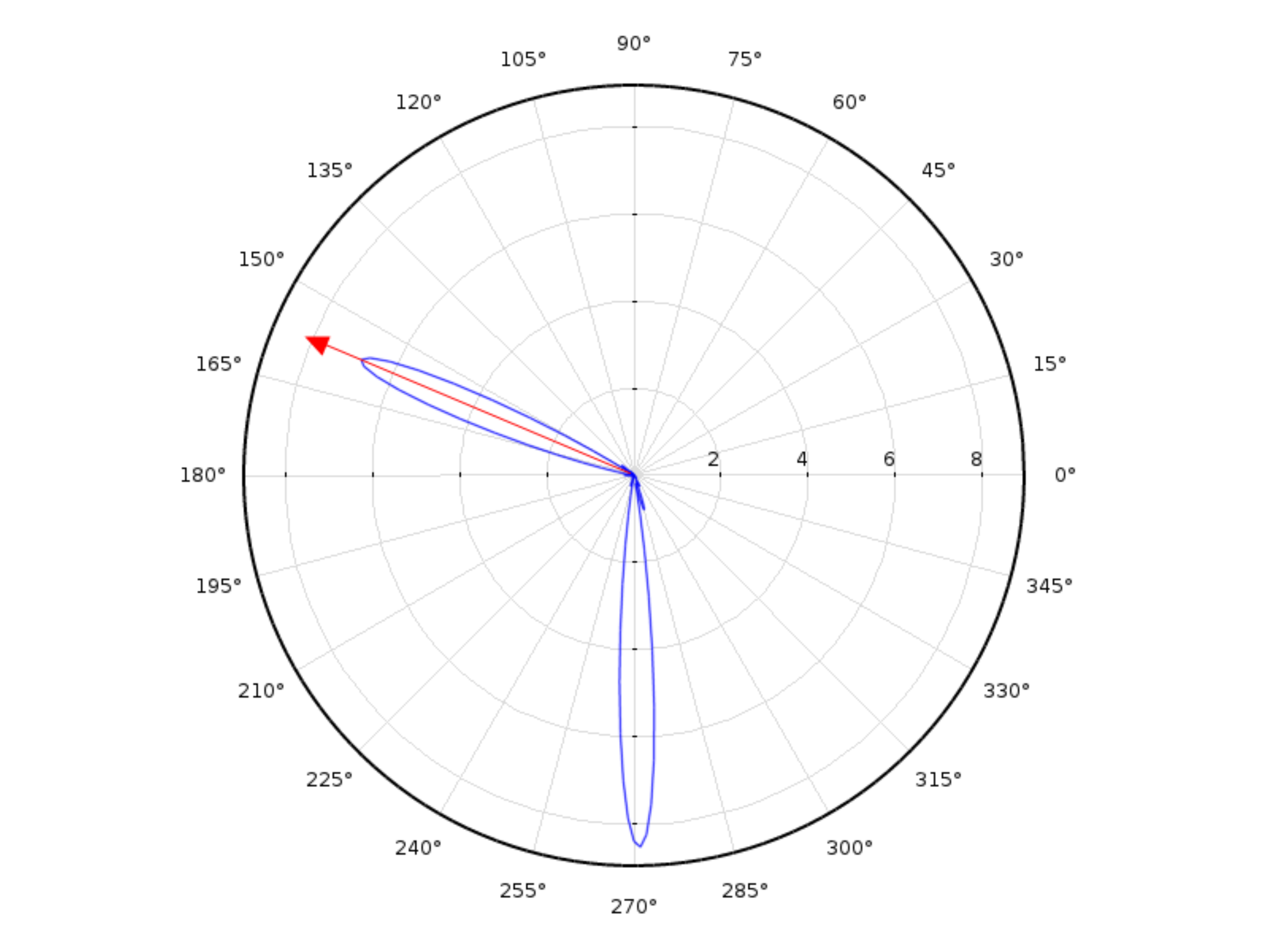}\hfill{} \includegraphics[width=0.43\textwidth]{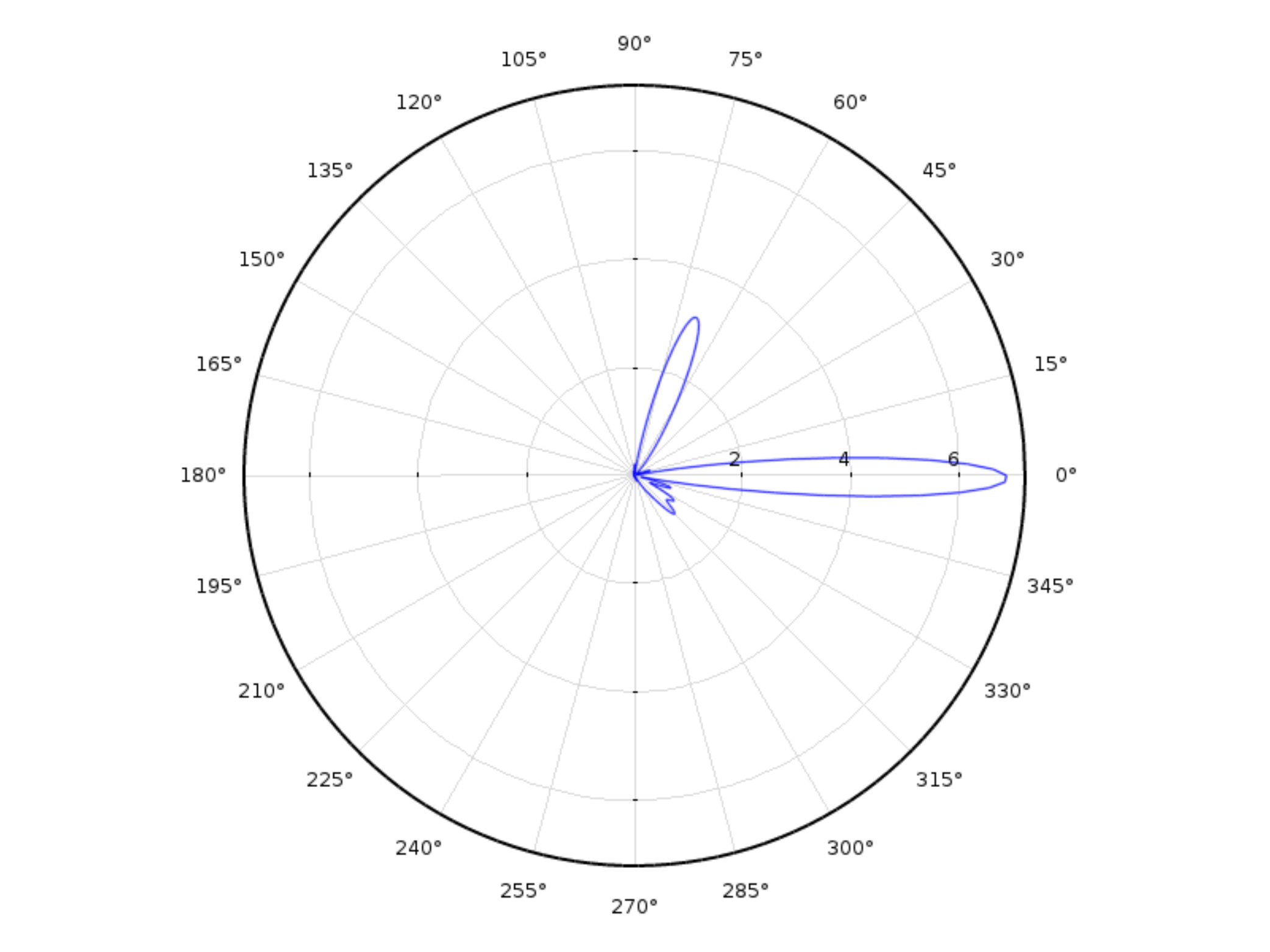}\hfill{}

\hfill{}(a)\hfill{}\hfill{}(b)\hfill{}

\hfill{}\includegraphics[width=0.43\textwidth]{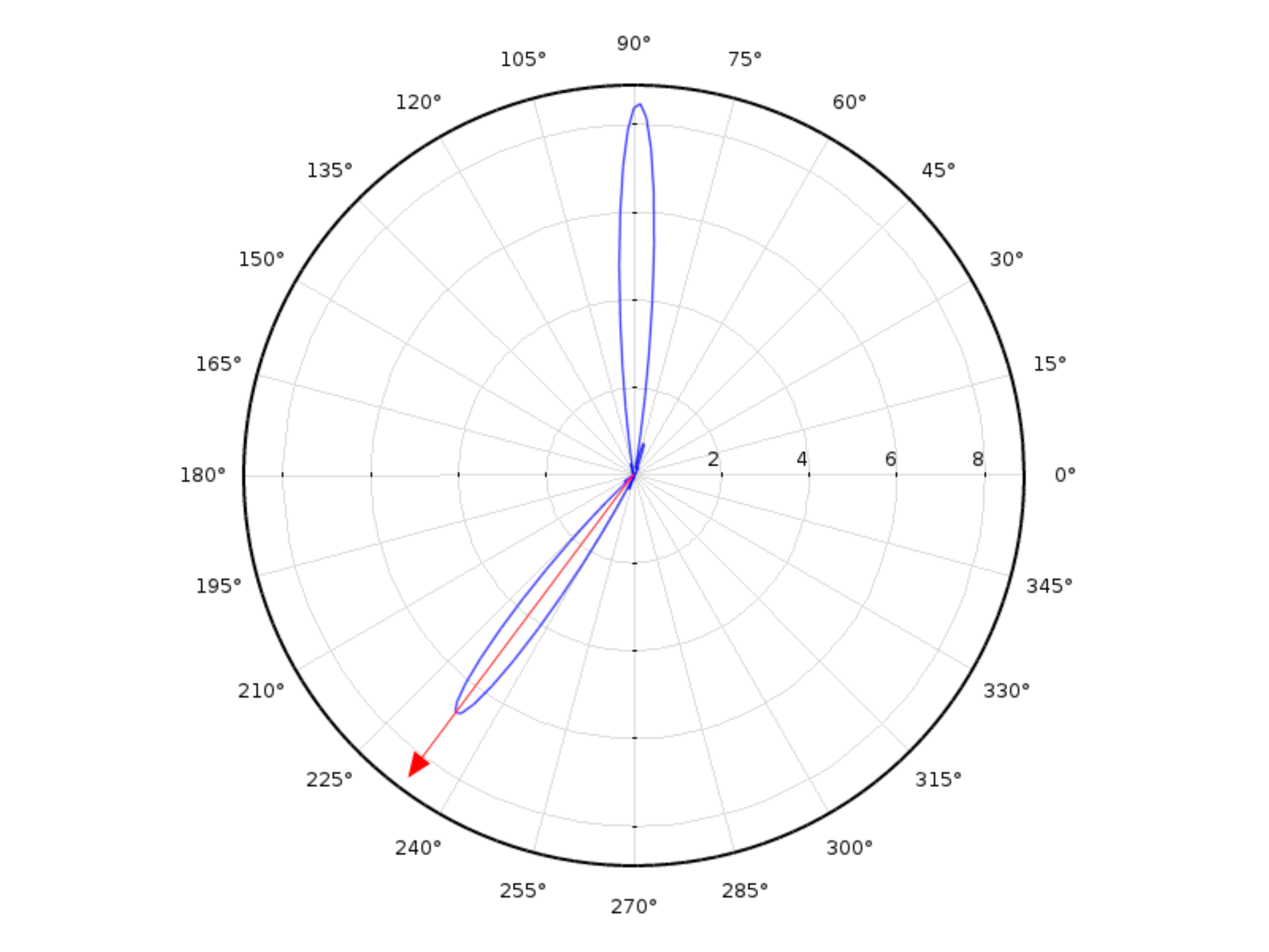}\hfill{} \includegraphics[width=0.43\textwidth]{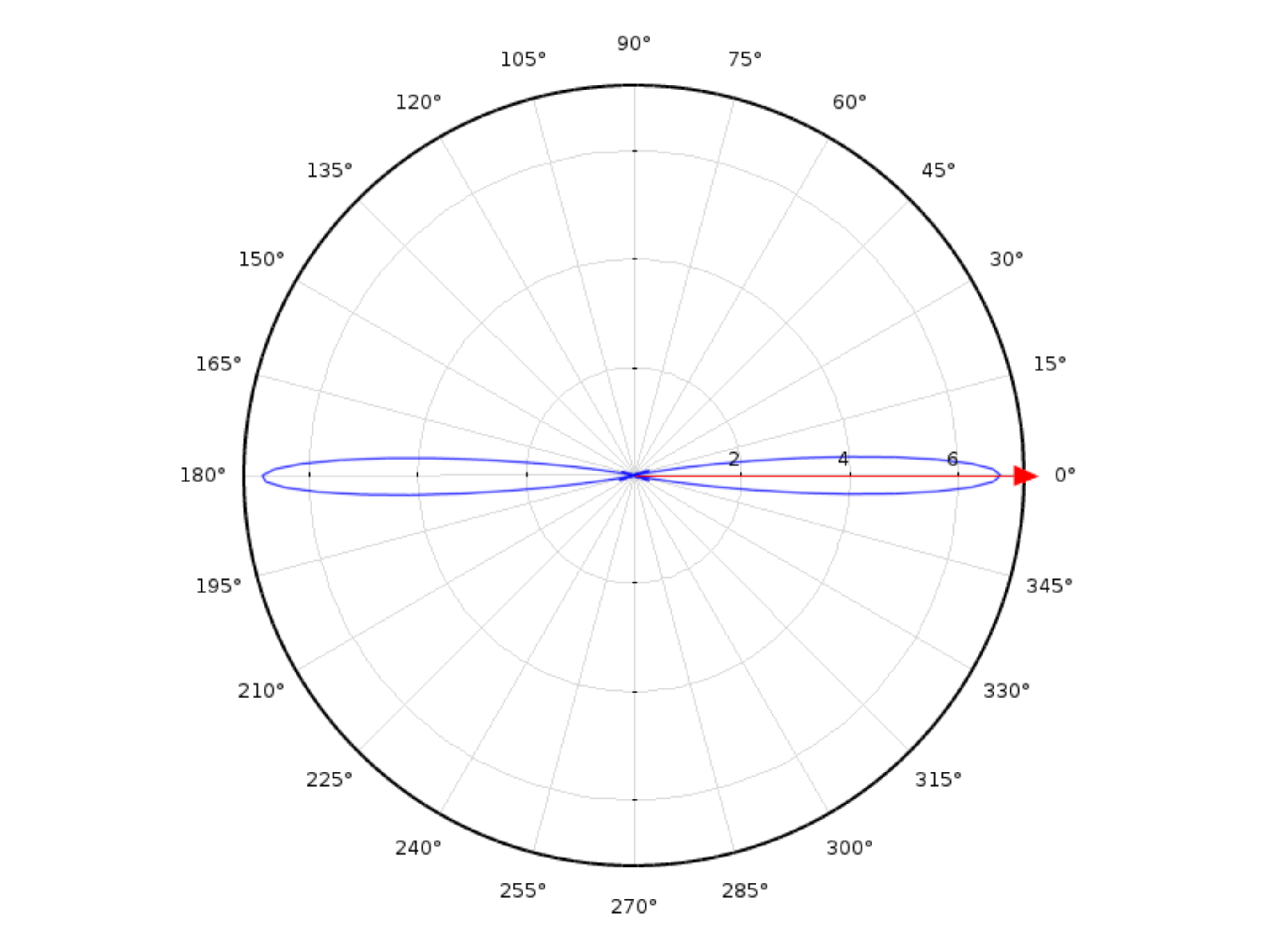}\hfill{}

\hfill{}(c)\hfill{}\hfill{}(d)\hfill{}

\caption{\label{fig:ex1:SS:ff} Plot of $|u^\infty(\hat x)|^2$ in polar coordinates corresponding to a sound-soft triangle due to an incident plane wave $e^{ikx\cdot d}$ with  $k=6\pi$ and $d=d_2,d_4,d_6,d_8$ from (a) to (d). The selected critical observation angles $\hat z_1$, $\hat z_2$ and $\hat z_3$ are highlighted by red arrows. }
\end{figure}

\medskip
\noindent {\bf Example 1. A triangle.}
\medskip

The obstacle is chosen to be a triangle with three vertices displaced at $(1,0)$, $(2.5,-0.5)$ and $(2.5,1)$, respectively. In this test,
we send off four detecting waves from north, east, south and east with $d=d_2,\ d_4,\ d_6,\ d_8$, respectively.

Firstly, we test a sound-soft triangular obstacle. The initial position $x_0$ of the polygonal scatterer is detected by Scheme I  (see \cite{LLZ}) and obtained
by taking the position with maximum indicator function value.
The initial guess of the location point is found to be $x_0=(2.136,\ 0.217)$, denoted by a red star in Fig.~\ref{fig:ex1:SS:reconstruction}.

We plot the square power of the phaseless far-field data in Fig.~\ref{fig:ex1:SS:ff}.  It can be seen from Fig.~\ref{fig:ex1:SS:ff} that
the phaseless data display significant maxima along the forward-scattering (or incident direction) directions within the shadow region of the obstacle in all four plots. Except the forward-scattering directions, we find several other directions with local maximum behavior in the polar plots and determine the three critical observation angles with the largest magnitude of phaseless data, which is indicated by the red arrows in Fig.~\ref{fig:ex1:SS:ff}(a), (c) and (d), respectively.
With the respective set of incident and critical observation directions obtained in hand, one can determine the three exterior unit normal directions of the sides of the polygonal obstacle.

The rest of the work is reduced to be a three-dimensional nonlinear least-squares minimization problem.  The final reconstruction results are shown in Fig.~\ref{fig:ex1:SS:reconstruction}, which is quite satisfactory  in both noise-free and noisy cases.
It can be observed that we can obtain better reconstruction by using detecting waves with relatively high wave number for this sound-soft polygonal scatterer. This can be explained in the following way.  The detecting wave with high frequency  gives a more focused reflected beam and thus yields a better approximation of the physical optics. The higher the frequency of the detecting wave, the more accurate determination of the outward normal directions and thus we can obtain better reconstruction plots. Moreover, we see that the proposed recovery scheme is tolerable to relatively high level of noise and performs robust in the noisy case.

\begin{figure}[h]
\centering
\hfill{}\includegraphics[width=0.48\textwidth]{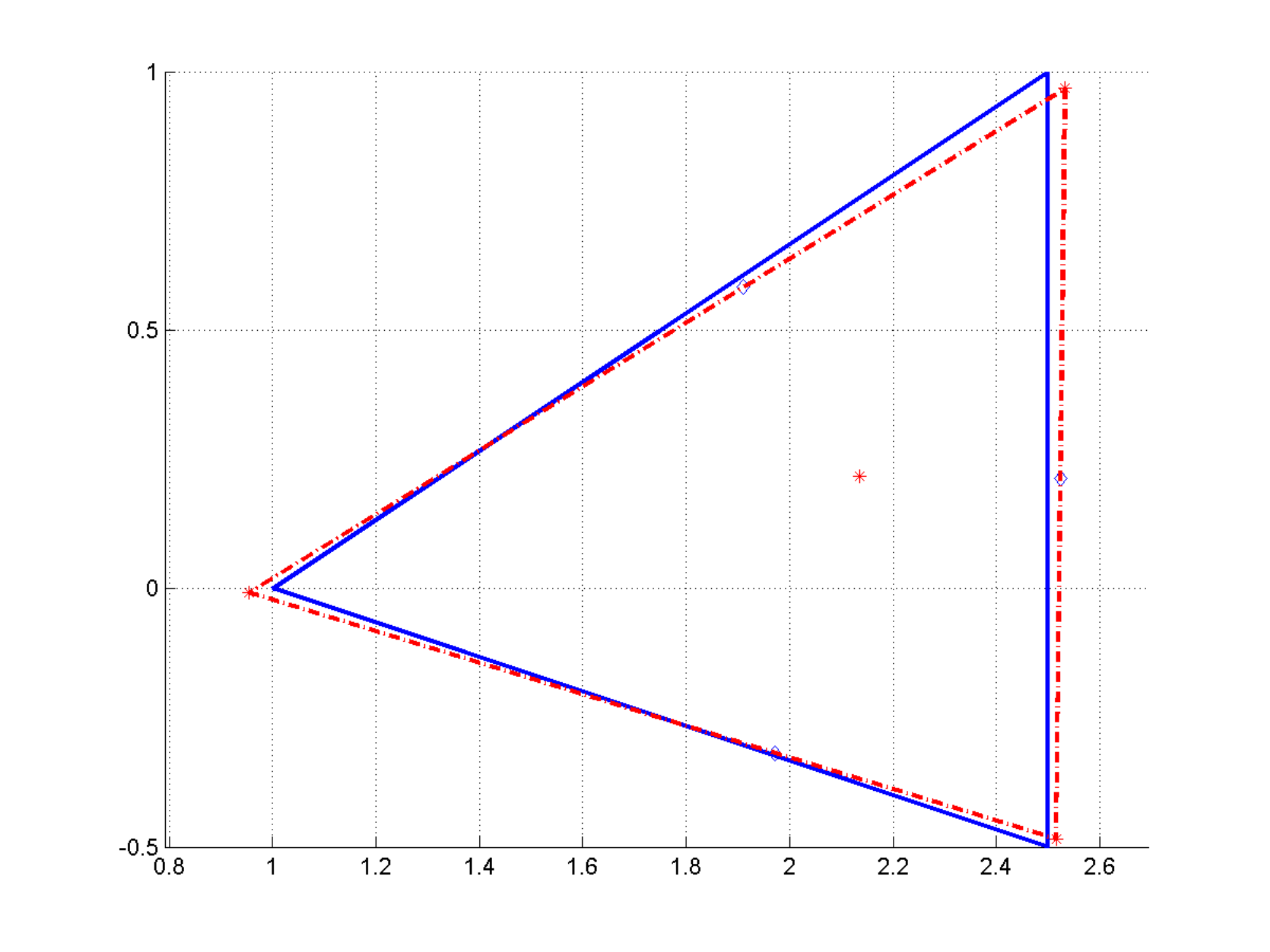}\hfill{} \includegraphics[width=0.48\textwidth]{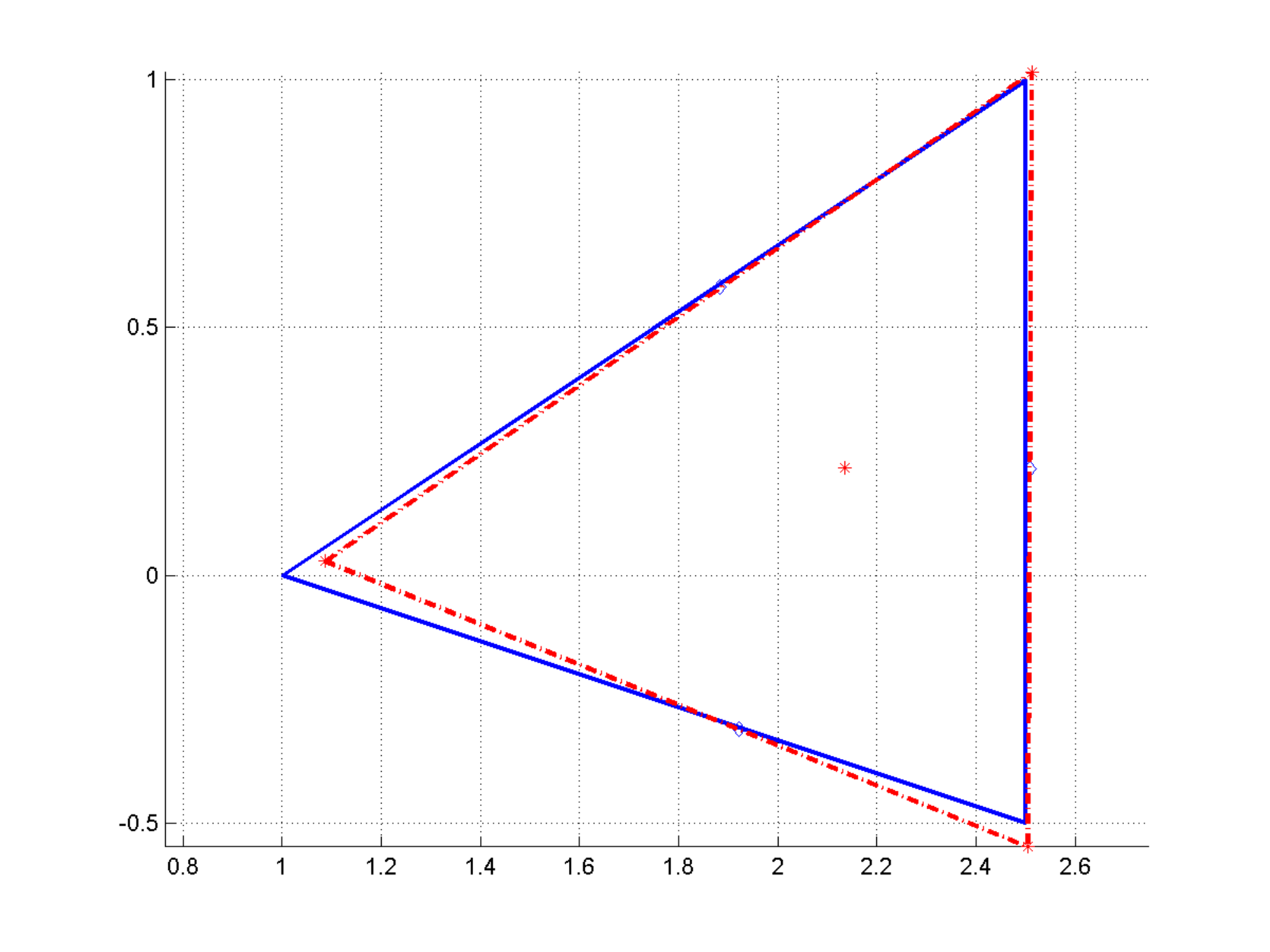}\hfill{}

\hfill{}(a)\hfill{}\hfill{}(b)\hfill{}

\hfill{}\includegraphics[width=0.48\textwidth]{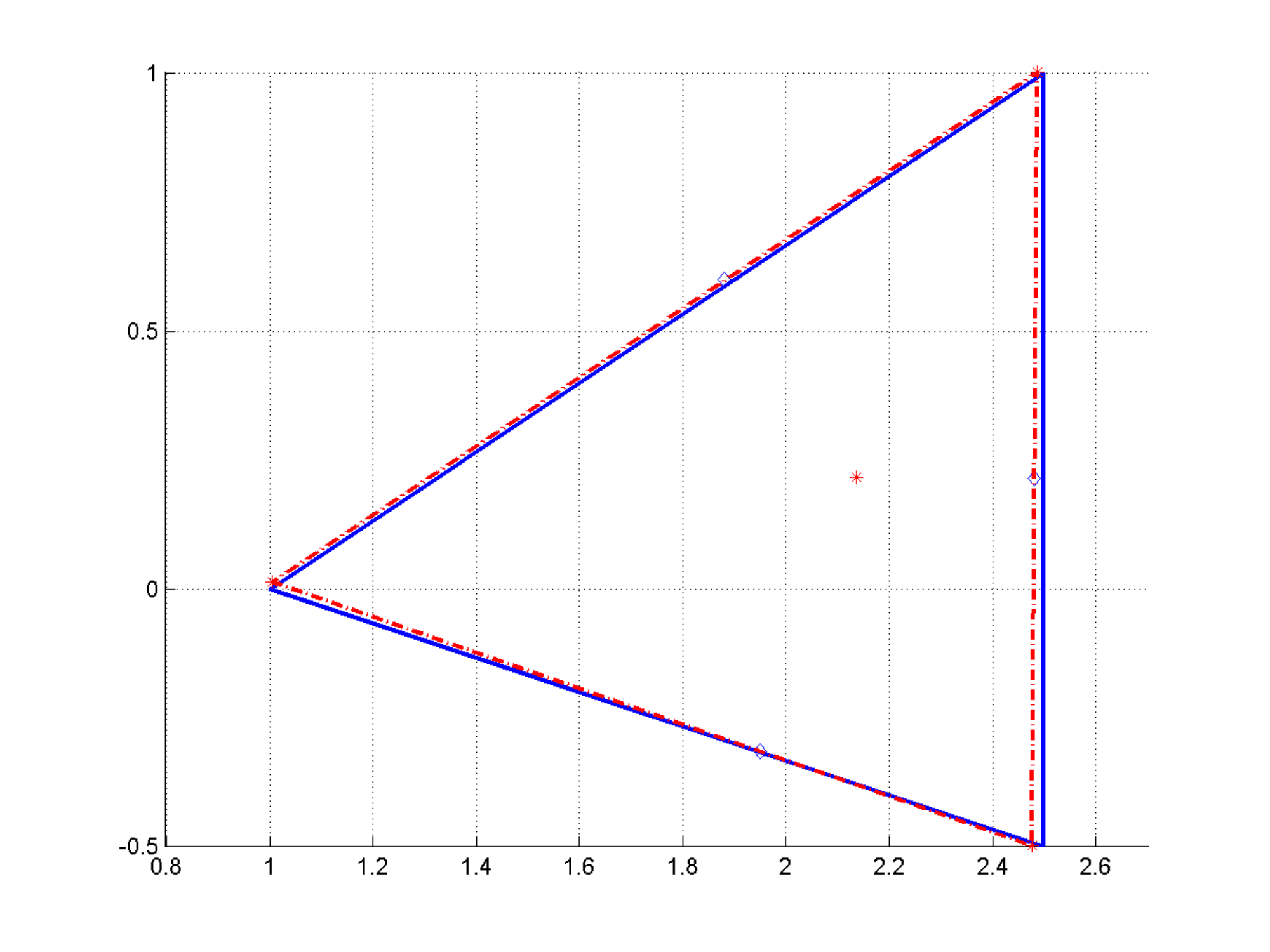}\hfill{} \includegraphics[width=0.48\textwidth]{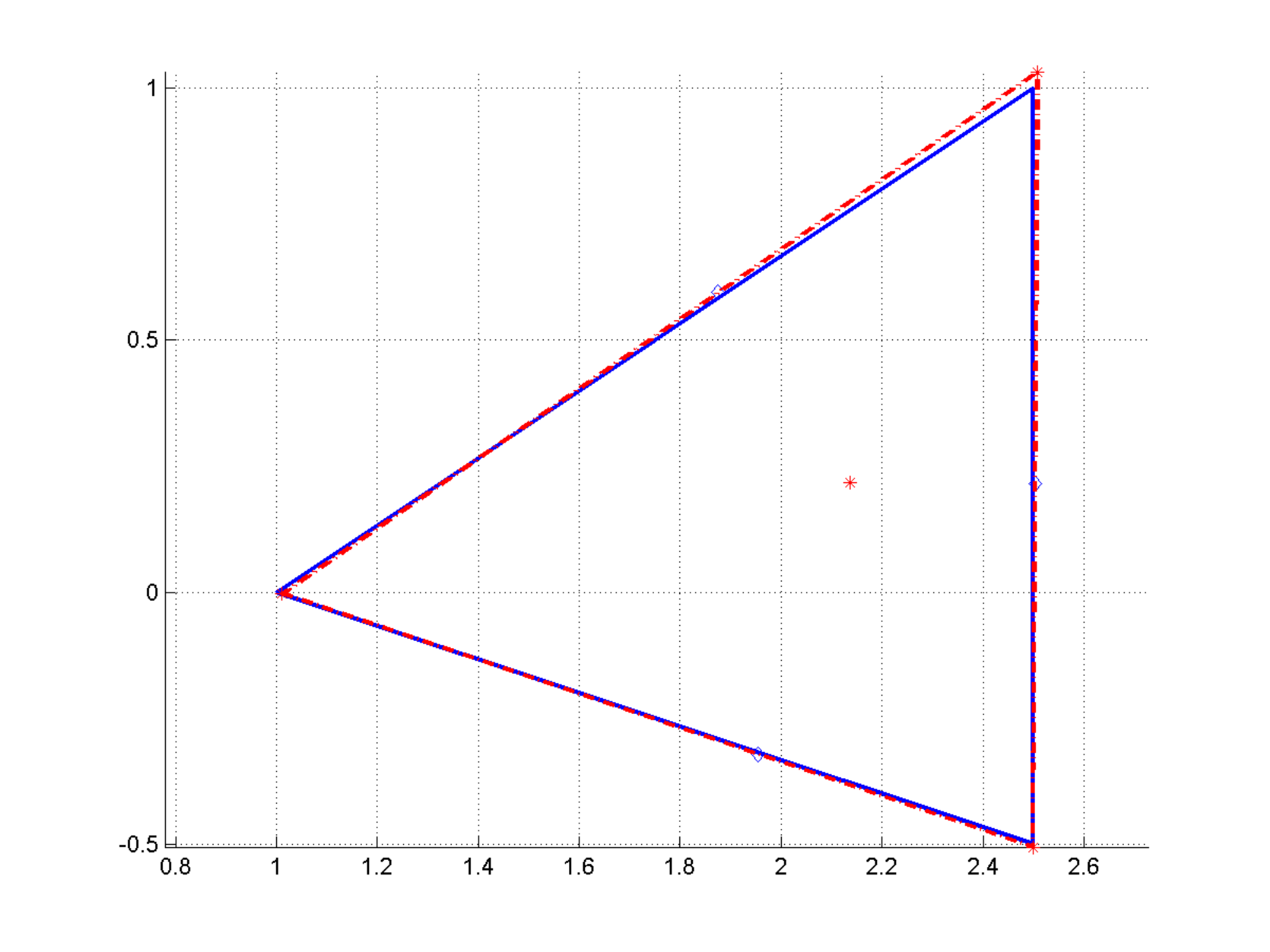}\hfill{}

\hfill{}(c)\hfill{}\hfill{}(d)\hfill{}

\caption{\label{fig:ex1:SS:reconstruction} Reconstruction of the triangular sound-soft scatterer with
(a) $k=6\pi$ with $5\%$ noise,  (b) $k=6\pi$  without noise,
(c) $k=10\pi$ with $5\%$ noise and (d) $k=10\pi$  without noise.  }
\end{figure}

Next, we keep the experimental settings unchanged except replacing the obstacle by a sound-hard triangular scatterer. As before, the location point is detected to be $x_0=(1.9307,\ 0.1412)$ using Scheme I in \cite{LLZ}, denoted by a red star in Fig.~\ref{fig:ex1:SH:reconstruction}. As shown in Fig.~\ref{fig:ex1:SH:ff} we plot the square power of the phaseless far-field data and indicate the first three critical observation angles within the backscattering aperture.
The final reconstruction results are shown in Fig.~\ref{fig:ex1:SH:reconstruction}, which performs as good as in the sound-soft case. The effectiveness of the recovery scheme can be explained in a similar way as in the sound-soft case.

\begin{figure}[tbh]
\centering
\hfill{}\includegraphics[width=0.43\textwidth]{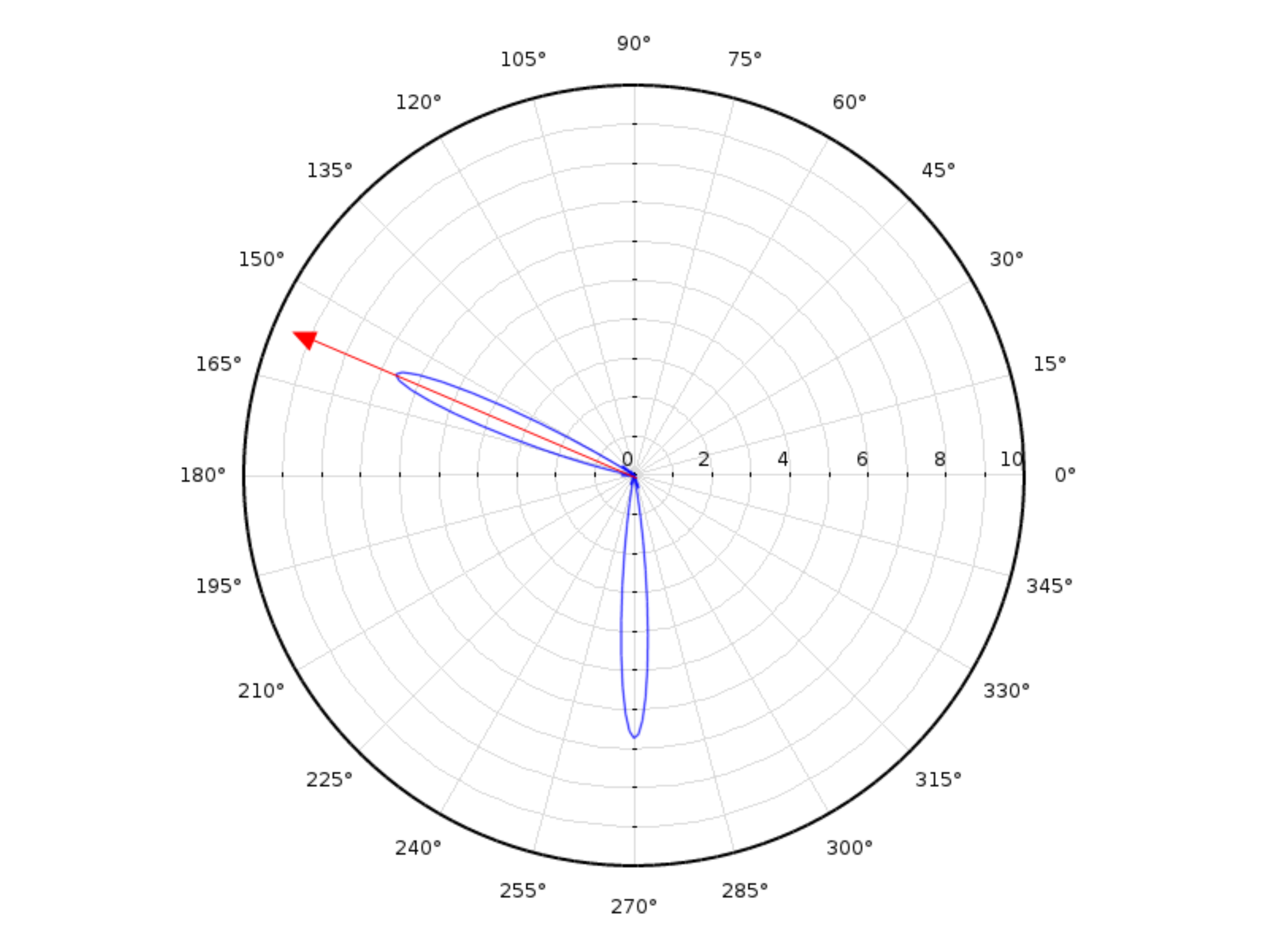}\hfill{} \includegraphics[width=0.43\textwidth]{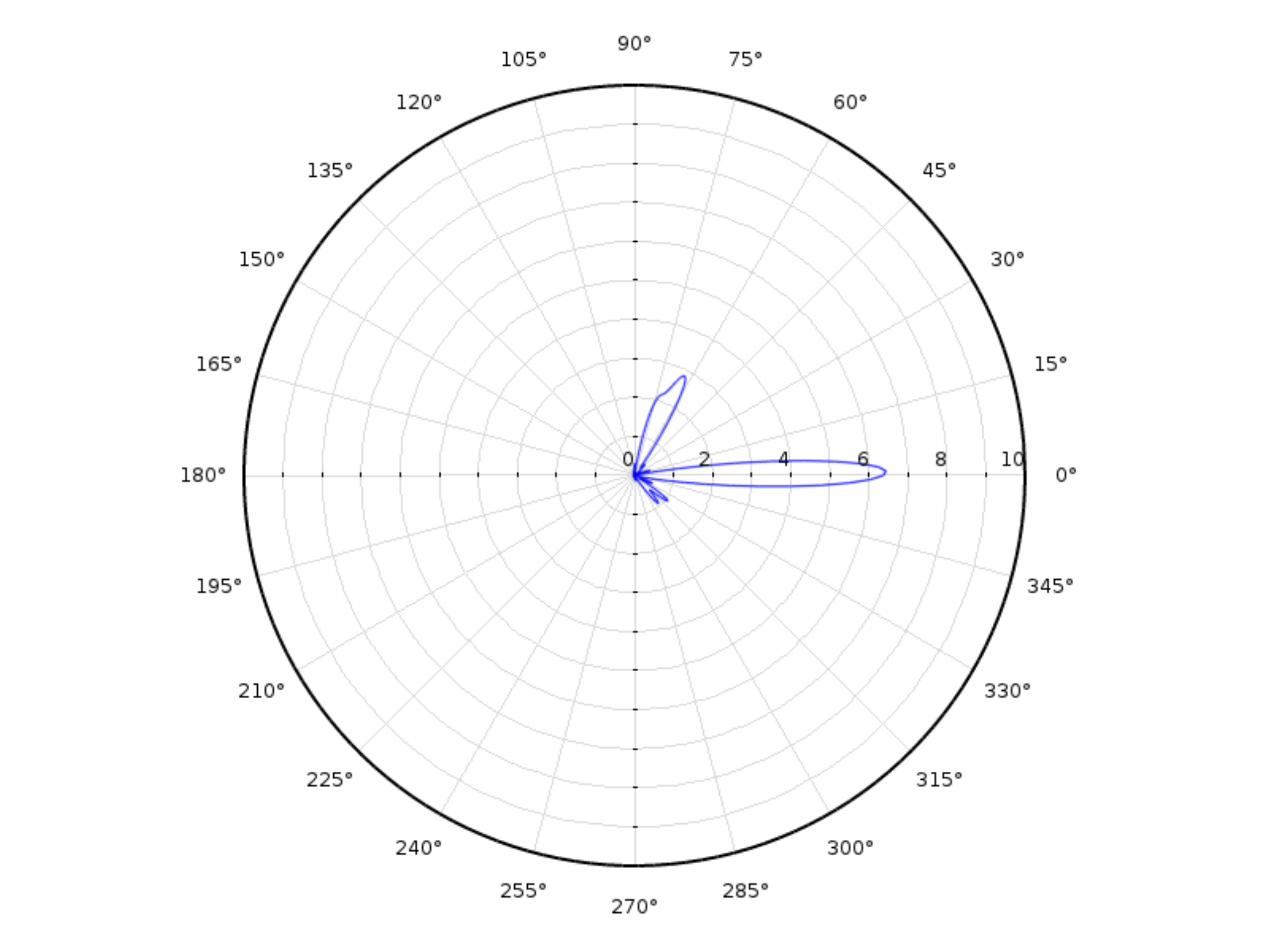}\hfill{}

\hfill{}(a)\hfill{}\hfill{}(b)\hfill{}

\hfill{}\includegraphics[width=0.43\textwidth]{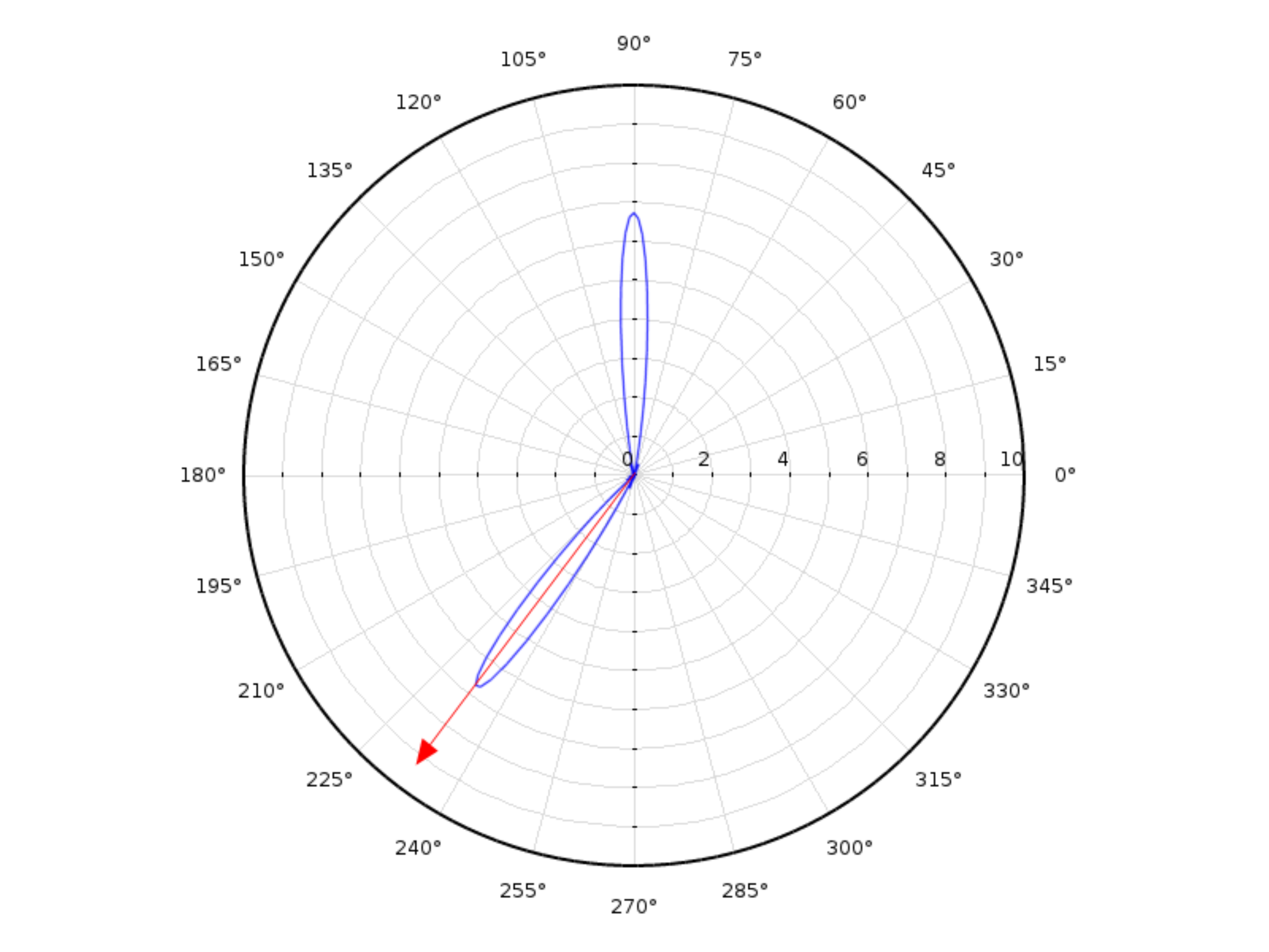}\hfill{} \includegraphics[width=0.43\textwidth]{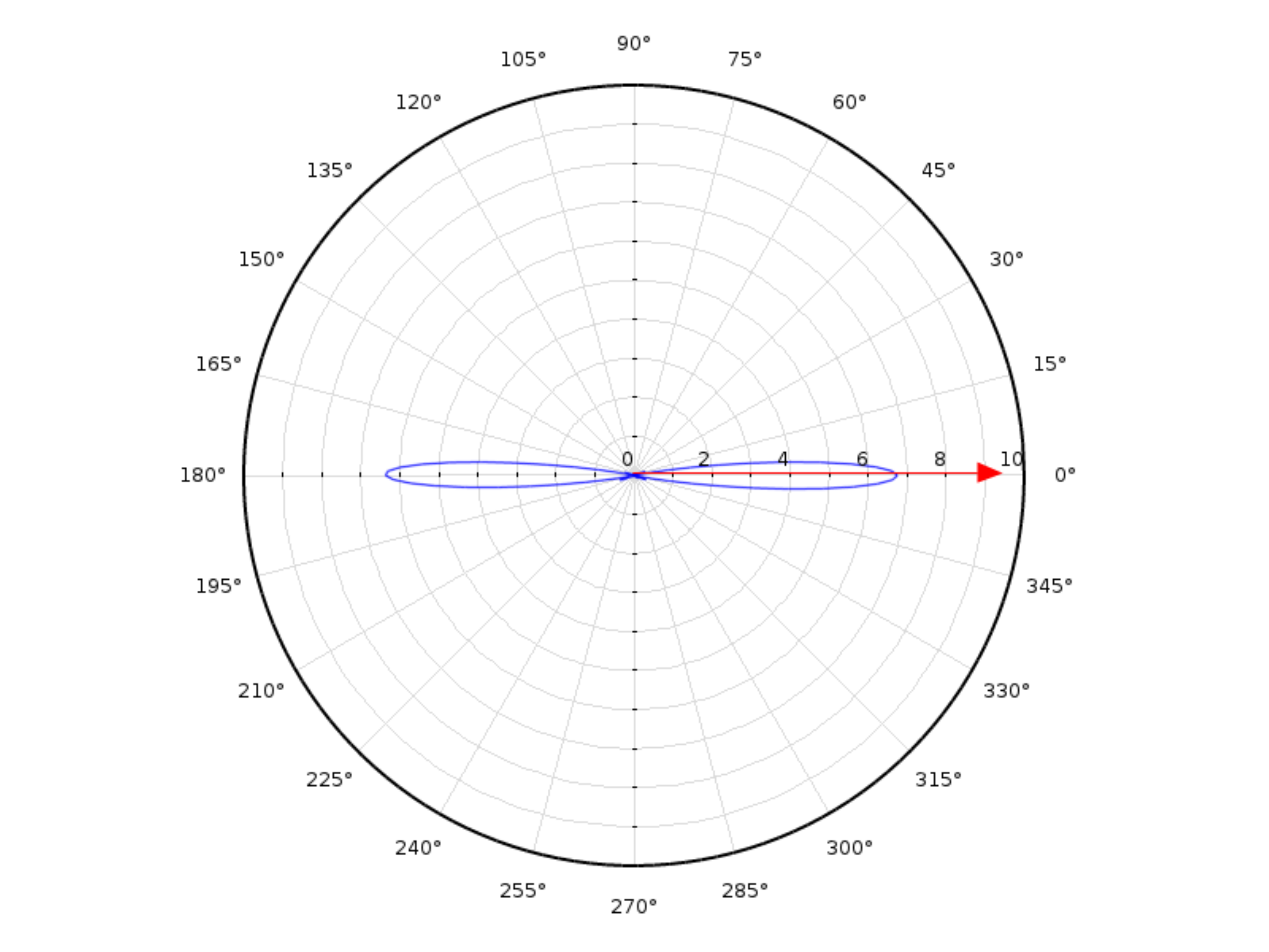}\hfill{}

\hfill{}(c)\hfill{}\hfill{}(d)\hfill{}

\caption{\label{fig:ex1:SH:ff} Plot of $|u^\infty(\hat x)|^2$ in polar coordinates corresponding to a sound-hard triangle due to an incident plane wave $e^{ikx\cdot d}$ with $k=6\pi$ and $d=d_2,d_4,d_6,d_8$ from (a) to (d). The selected critical observation directions $\hat z_1$, $\hat z_2$ and $\hat z_3$ are highlighted by red arrows. }
\end{figure}

\begin{figure}[tbh]
\centering
\hfill{}\includegraphics[width=0.48\textwidth]{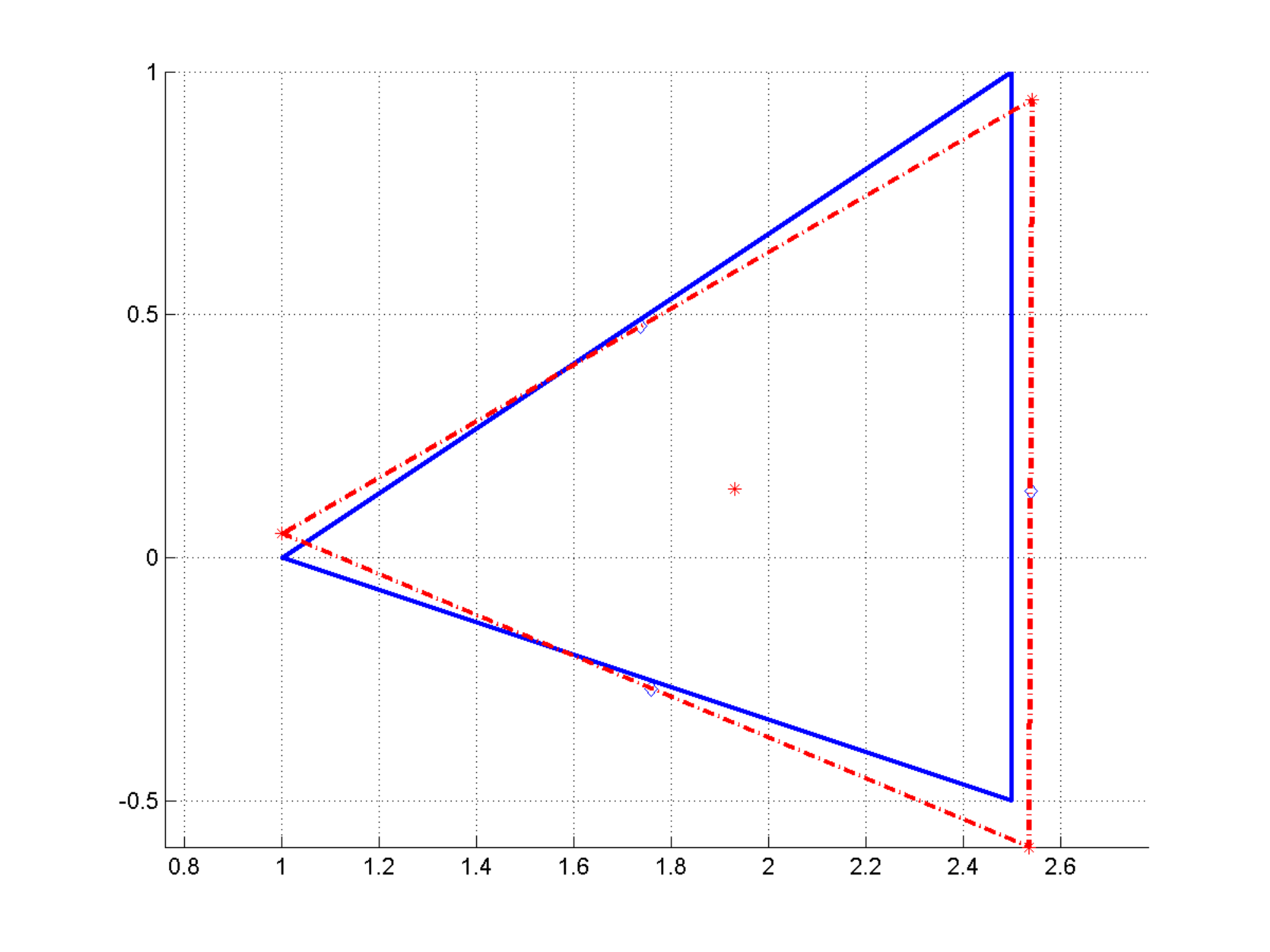}\hfill{} \includegraphics[width=0.48\textwidth]{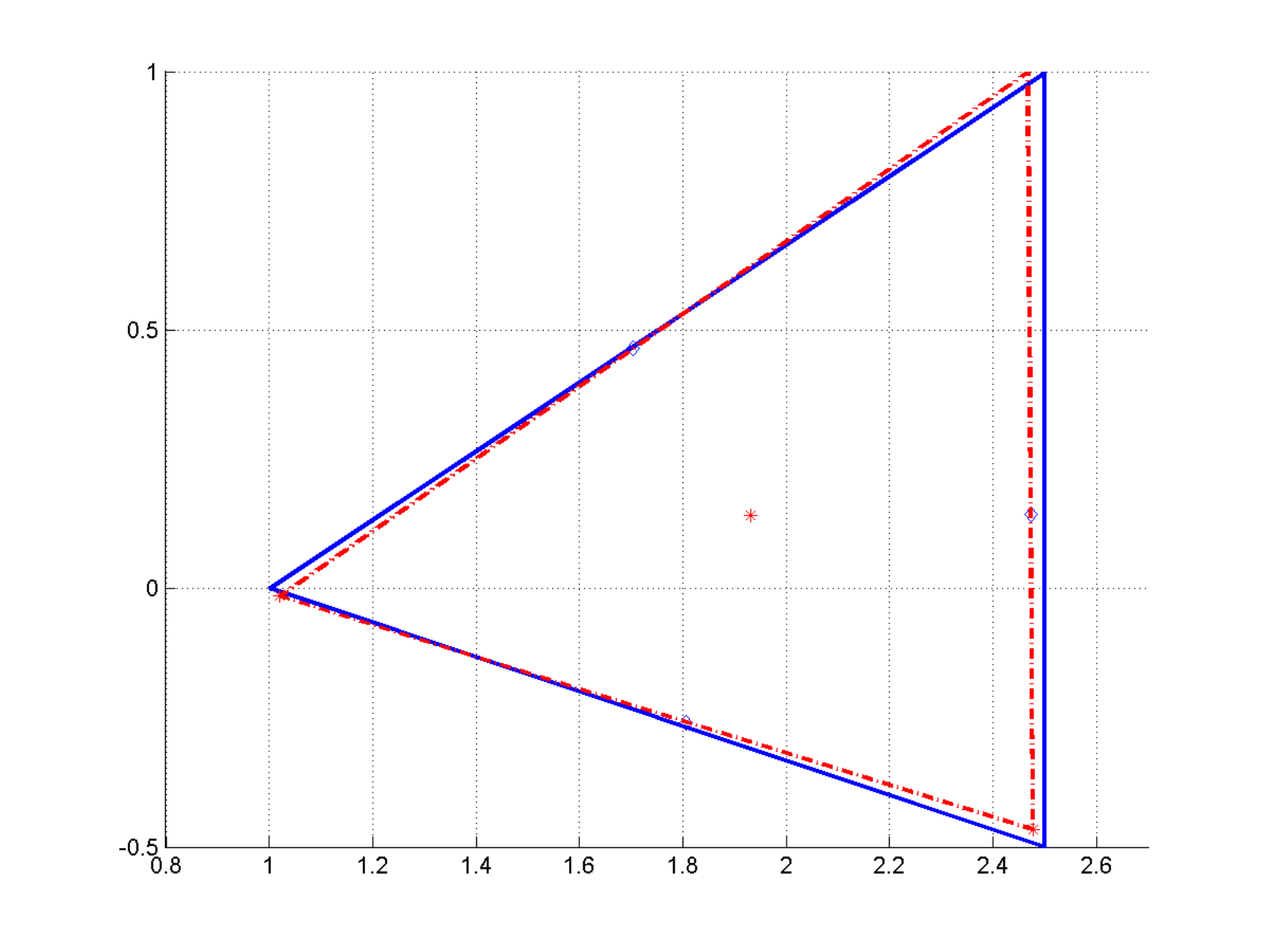}\hfill{}

\hfill{}(a)\hfill{}\hfill{}(b)\hfill{}

\hfill{}\includegraphics[width=0.48\textwidth]{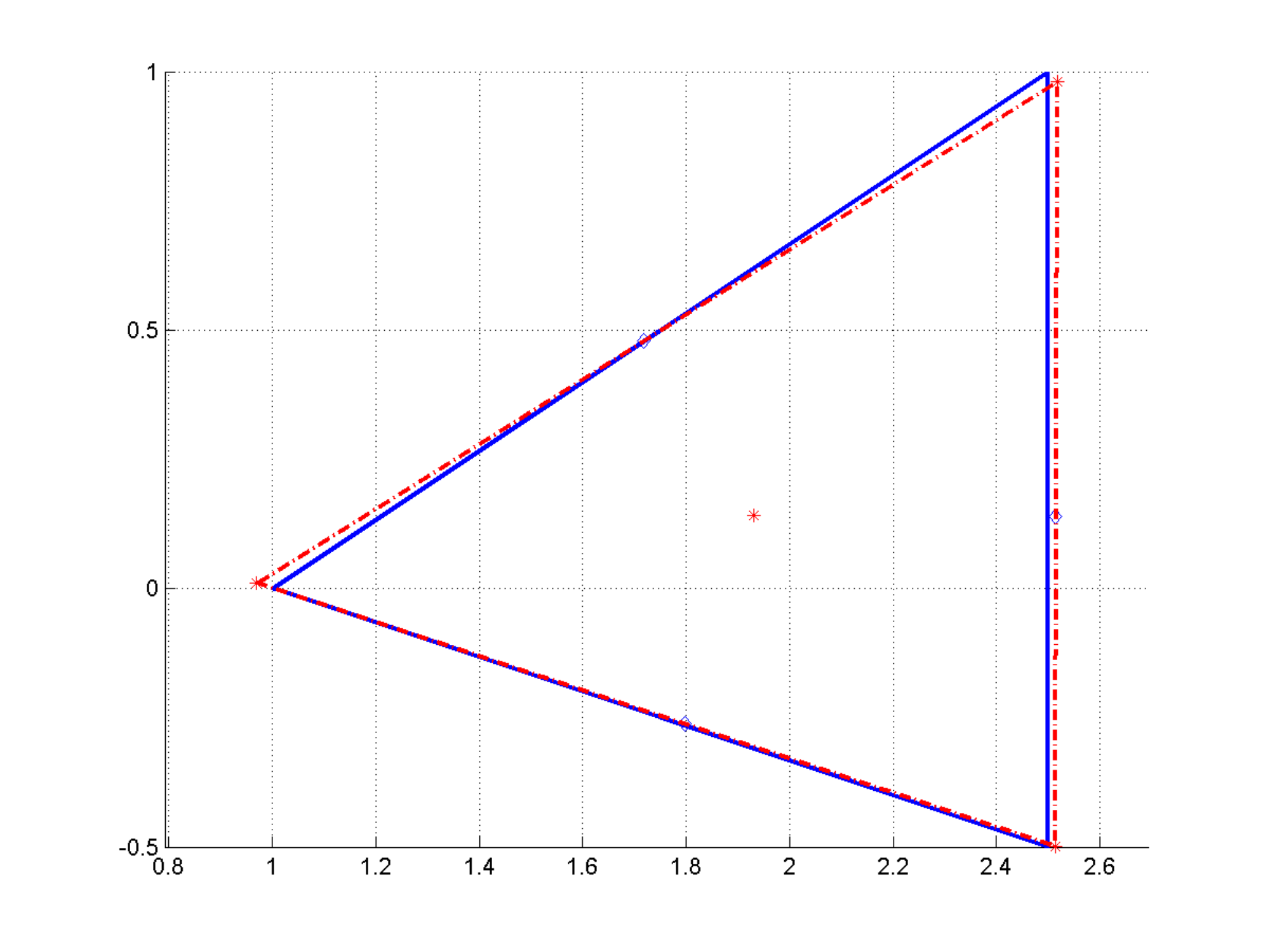}\hfill{} \includegraphics[width=0.48\textwidth]{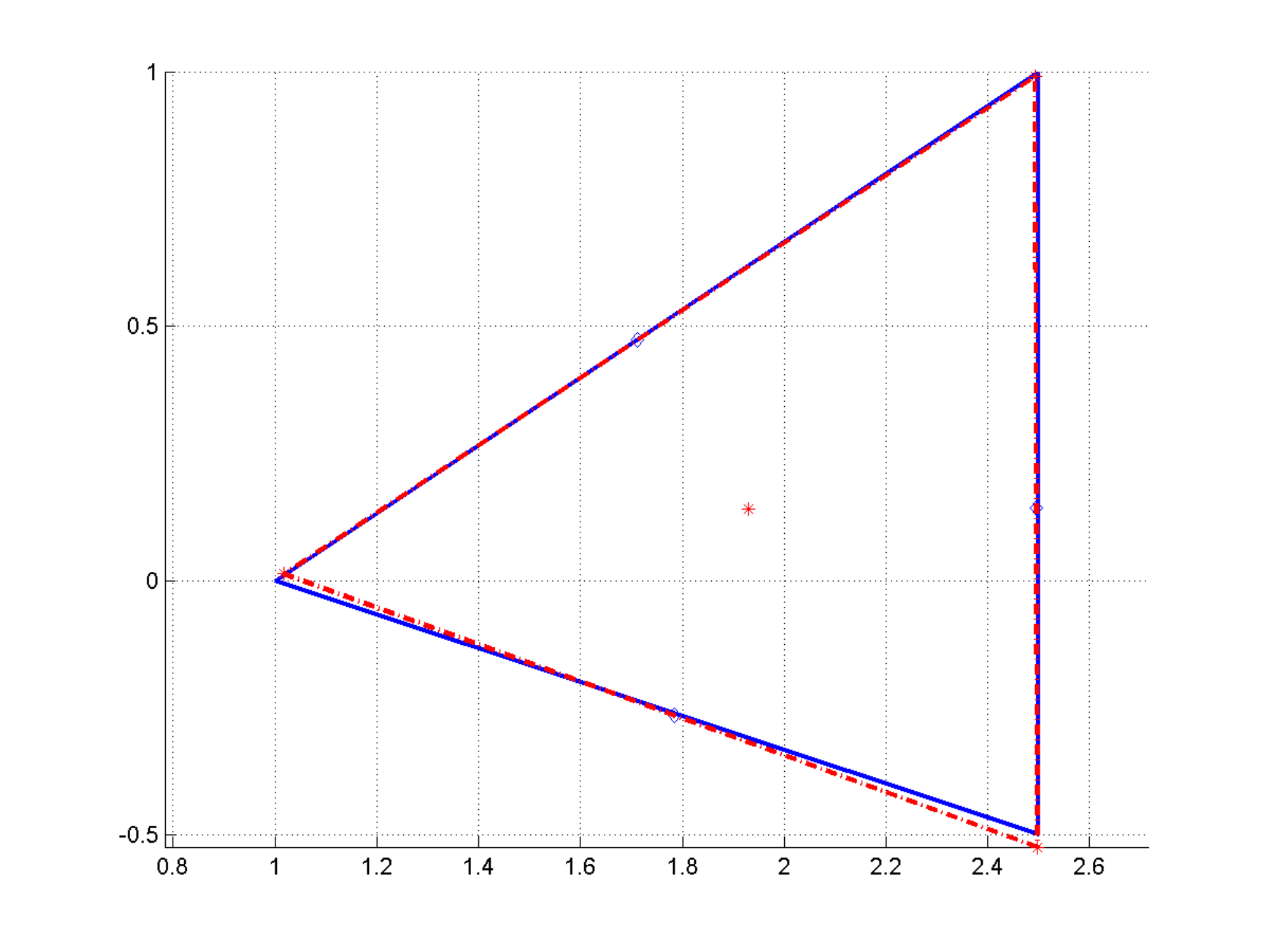}\hfill{}

\hfill{}(c)\hfill{}\hfill{}(d)\hfill{}

\caption{\label{fig:ex1:SH:reconstruction} Reconstruction of the triangular sound-hard scatterer with
(a) $k=6\pi$ with $5\%$ noise, (b) $k=6\pi$ without noise,
(c) $k=10\pi$ with $5\%$ noise, and (d) $k=10\pi$ without noise.}
\end{figure}

\medskip

\noindent {\bf Example 2. A convex hexagon. }

\medskip

In the second example, the obstacle is chosen to be a sound-soft hexagon with six vertices displaced at $(4,2.5)$, $(3,3)$ $(1,2)$, $(0.5,0)$ $(2,-1)$ and $(4.5,-0.5)$, respectively.
The location point is detected to be $x_0=(2.582,\ 0.759)$ using Scheme I in \cite{LLZ}. This example is much more challenging since there are multiple facets to be determined.

It is pointed out that the hexagonal obstacle has six sides and we only send off four detecting waves along incident directions $d=d_1,\ d_3,\ d_5,\ d_7$. In this case, the unknown number of sides is larger than that of incident directions. But from Fig.~\ref{fig:ex2:SS:ff}, one can reveal sufficient critical observation angles from those polar plots. We identify six critical observation angles with relatively large local maximum values  as indicated in red arrows in Fig.~\ref{fig:ex2:SS:ff}.  The final reconstruction results are obtained by solving a six-dimensional nonlinear problem and are shown in Fig.~\ref{fig:ex2:SS:reconstruction}. It is again observed that the reconstruction performs better with higher frequency detecting waves. With the combining effect of increasing sides and noise, our recovery scheme still yield quite reasonable approximation to the original hexagon.

\begin{figure}[tbh]
\centering
\hfill{}\includegraphics[width=0.43\textwidth]{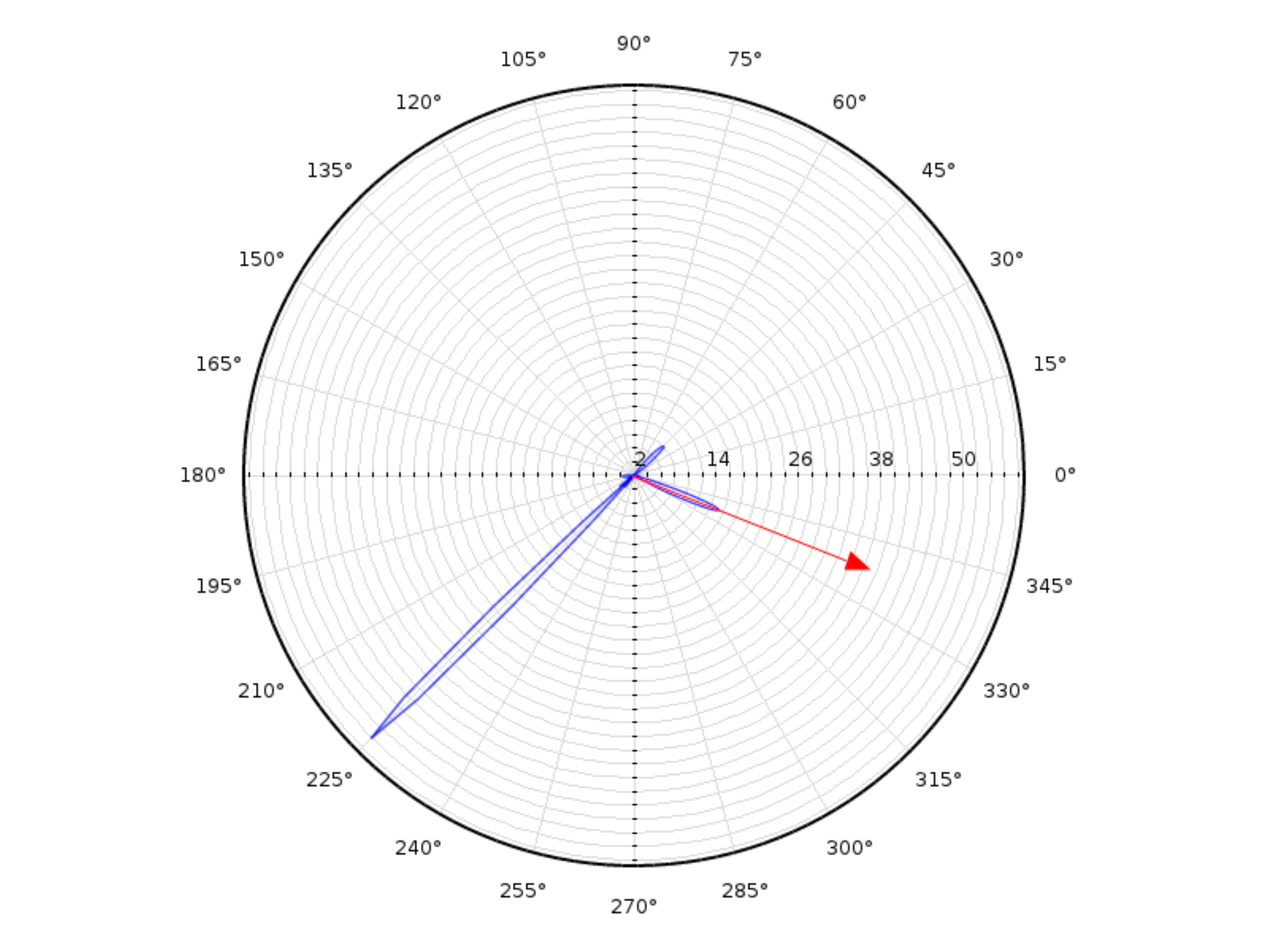}\hfill{} \includegraphics[width=0.43\textwidth]{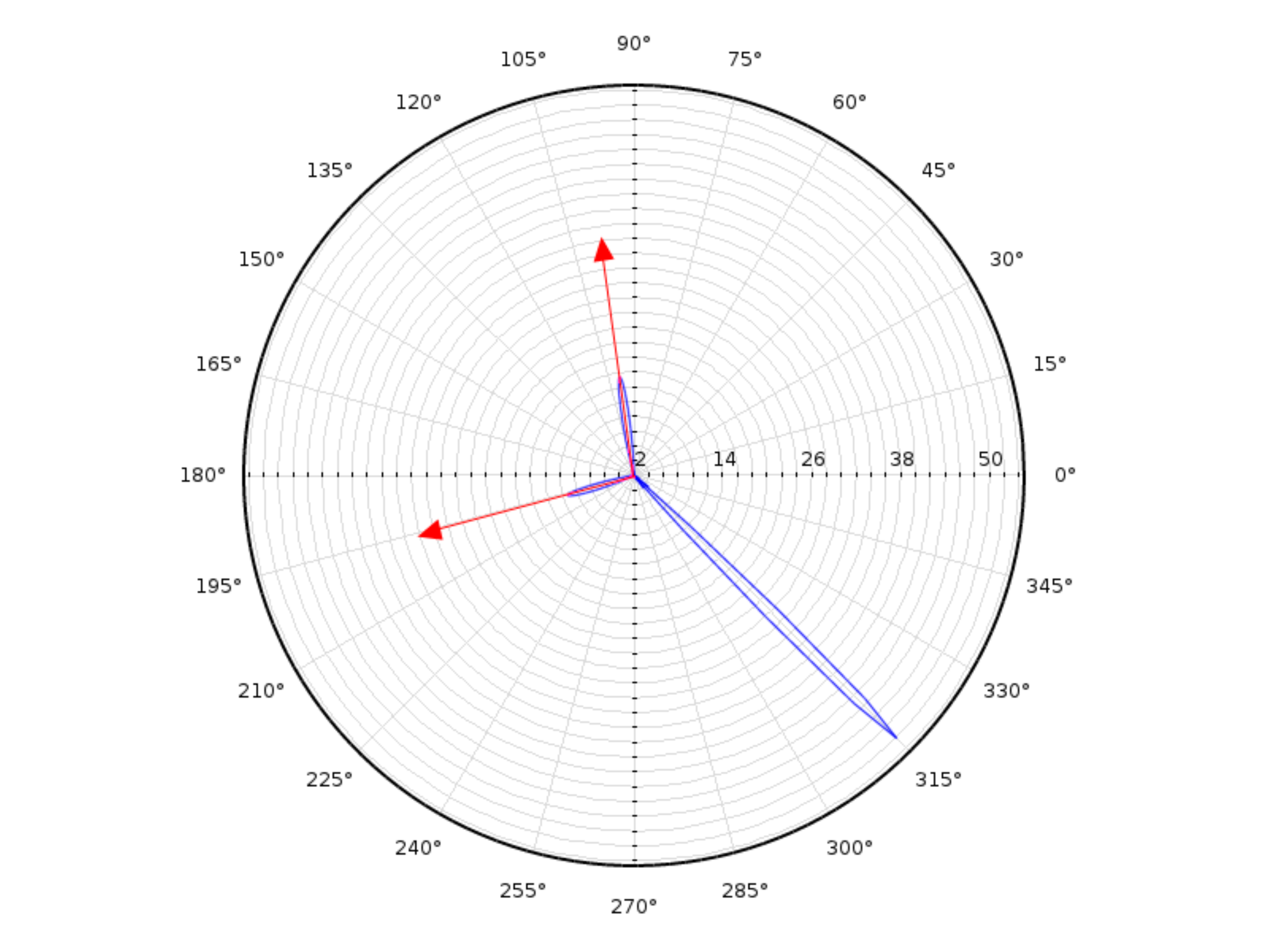}\hfill{}

\hfill{}(a)\hfill{}\hfill{}(b)\hfill{}

\hfill{}\includegraphics[width=0.43\textwidth]{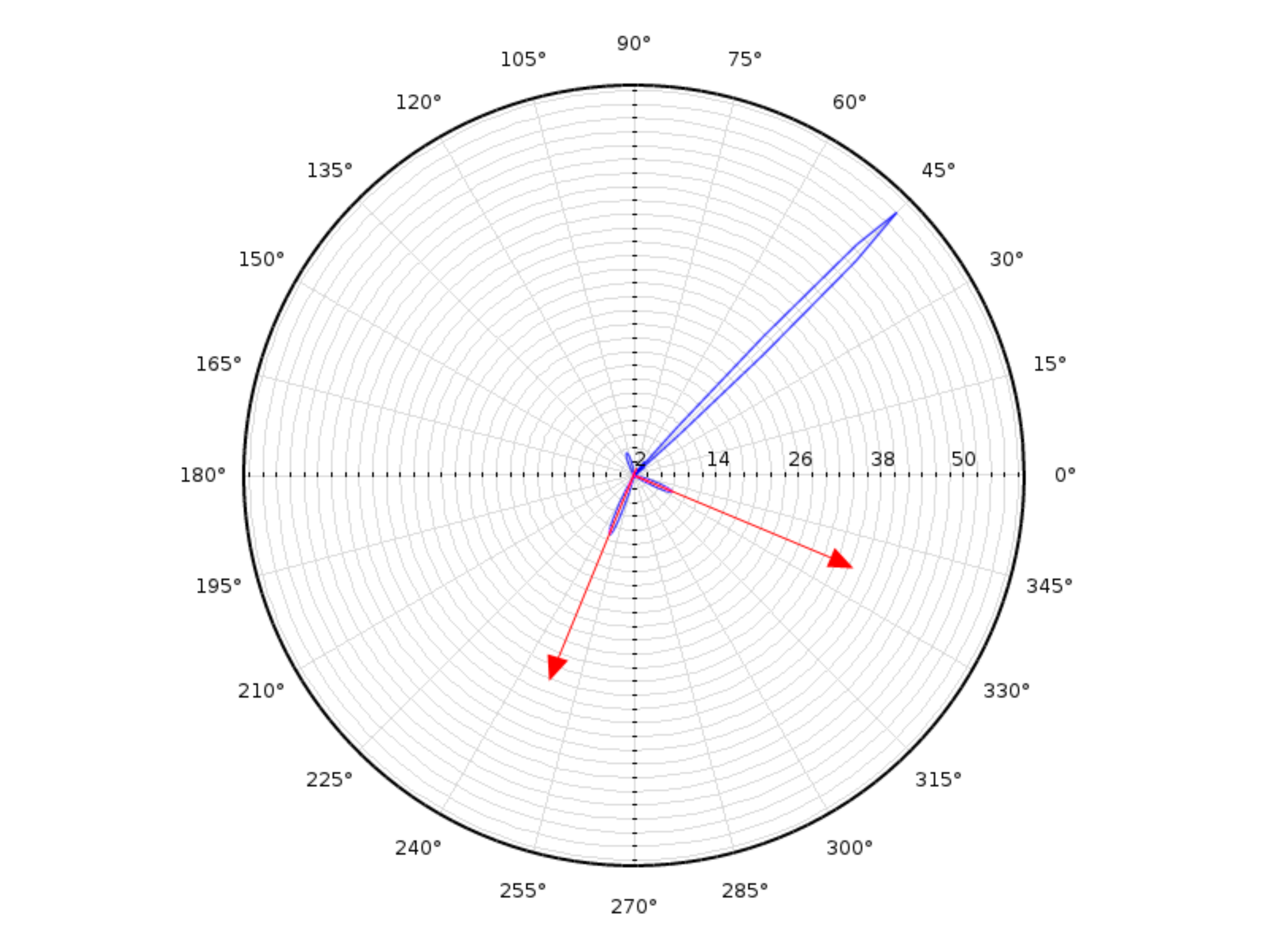}\hfill{} \includegraphics[width=0.43\textwidth]{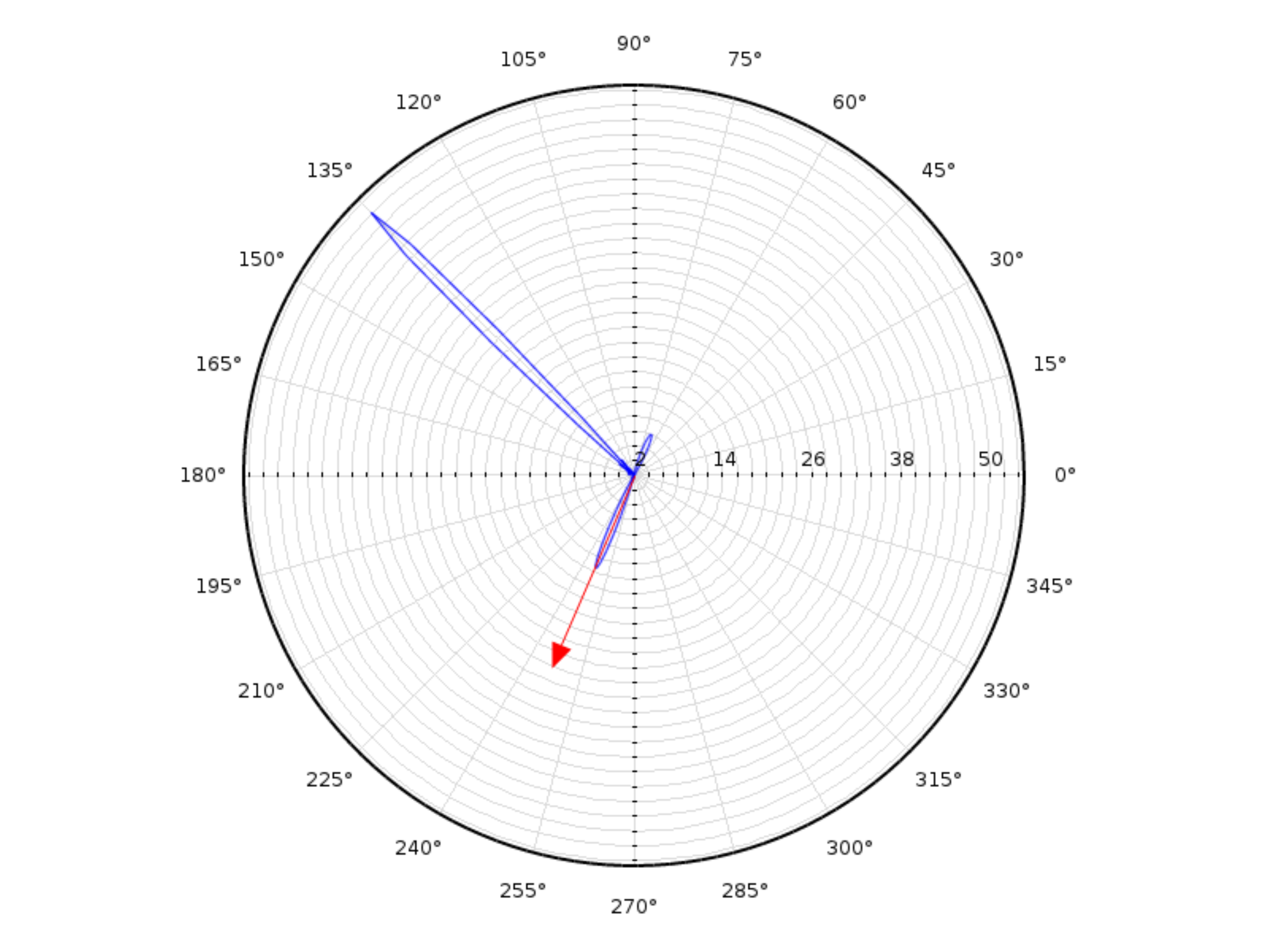}\hfill{}

\hfill{}(c)\hfill{}\hfill{}(d)\hfill{}

\caption{\label{fig:ex2:SS:ff} Plot of $|u^\infty(\hat x)|^2$ in polar coordinates corresponding to a convex sound-soft hexagon due to an incident plane wave $e^{ikx\cdot d}$ with  $k=6\pi$ and $d=d_1,d_3,d_5,d_7$ from (a) to (d). The selected critical observation directions $\hat z_1$, $\hat z_2$, $\hat z_3$, $\hat z_4$, $\hat z_5$ and $\hat z_6$ are highlighted by red arrows. }
\end{figure}

\begin{figure}[tbh]
\centering
\hfill{}\includegraphics[width=0.48\textwidth]{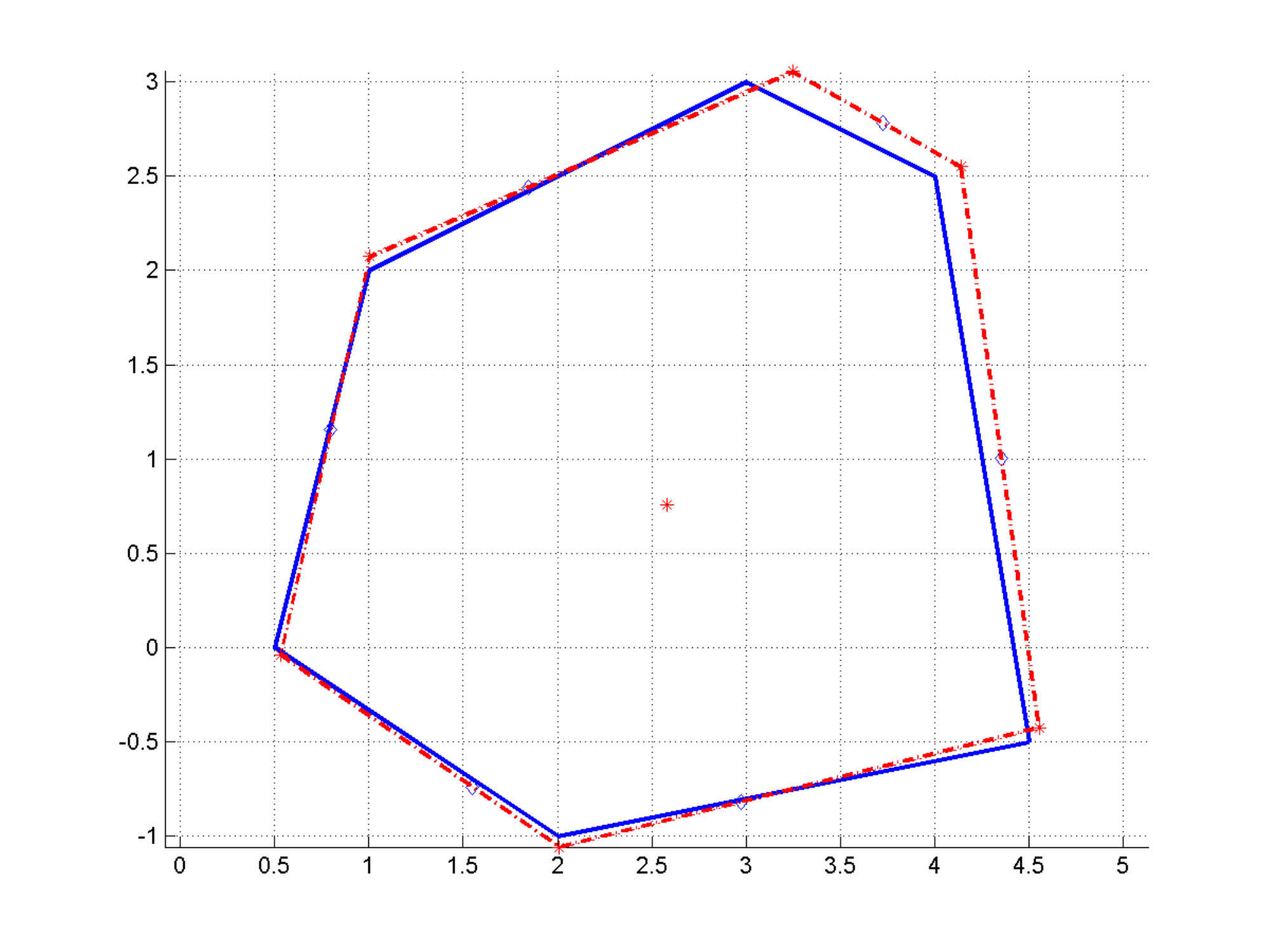}\hfill{} \includegraphics[width=0.48\textwidth]{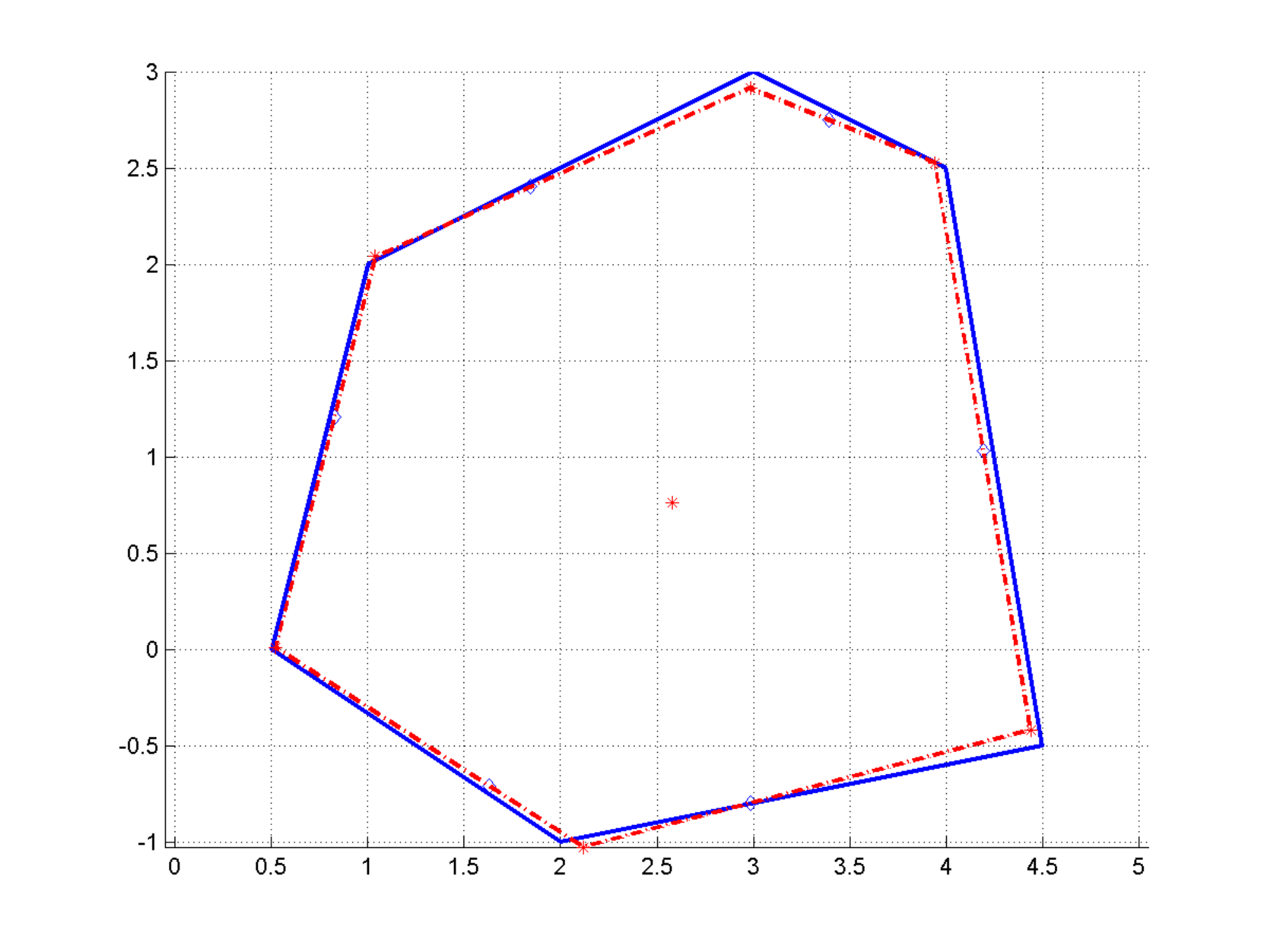}\hfill{}

\hfill{}(a)\hfill{}\hfill{}(b)\hfill{}

\hfill{}\includegraphics[width=0.48\textwidth]{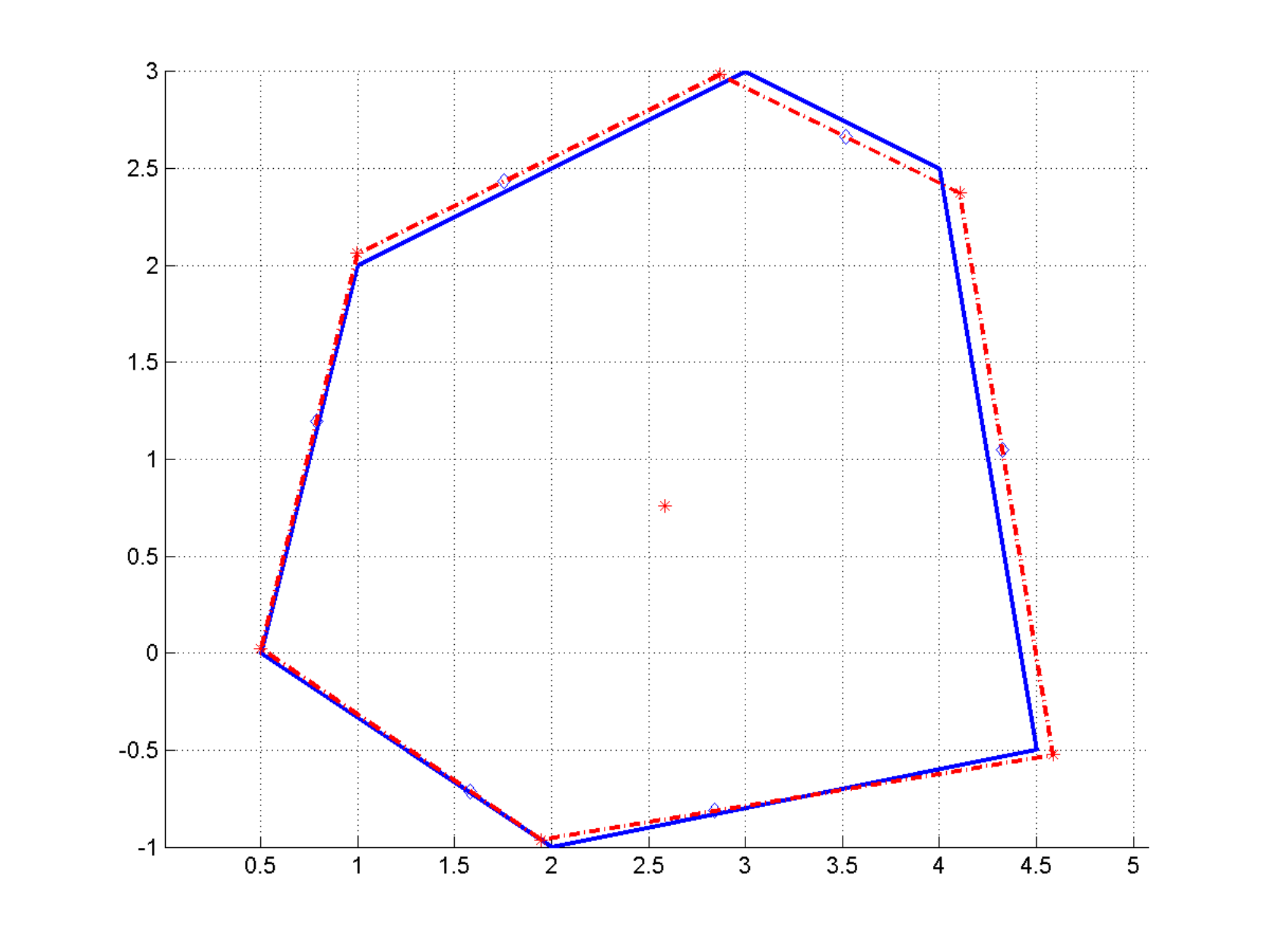}\hfill{} \includegraphics[width=0.48\textwidth]{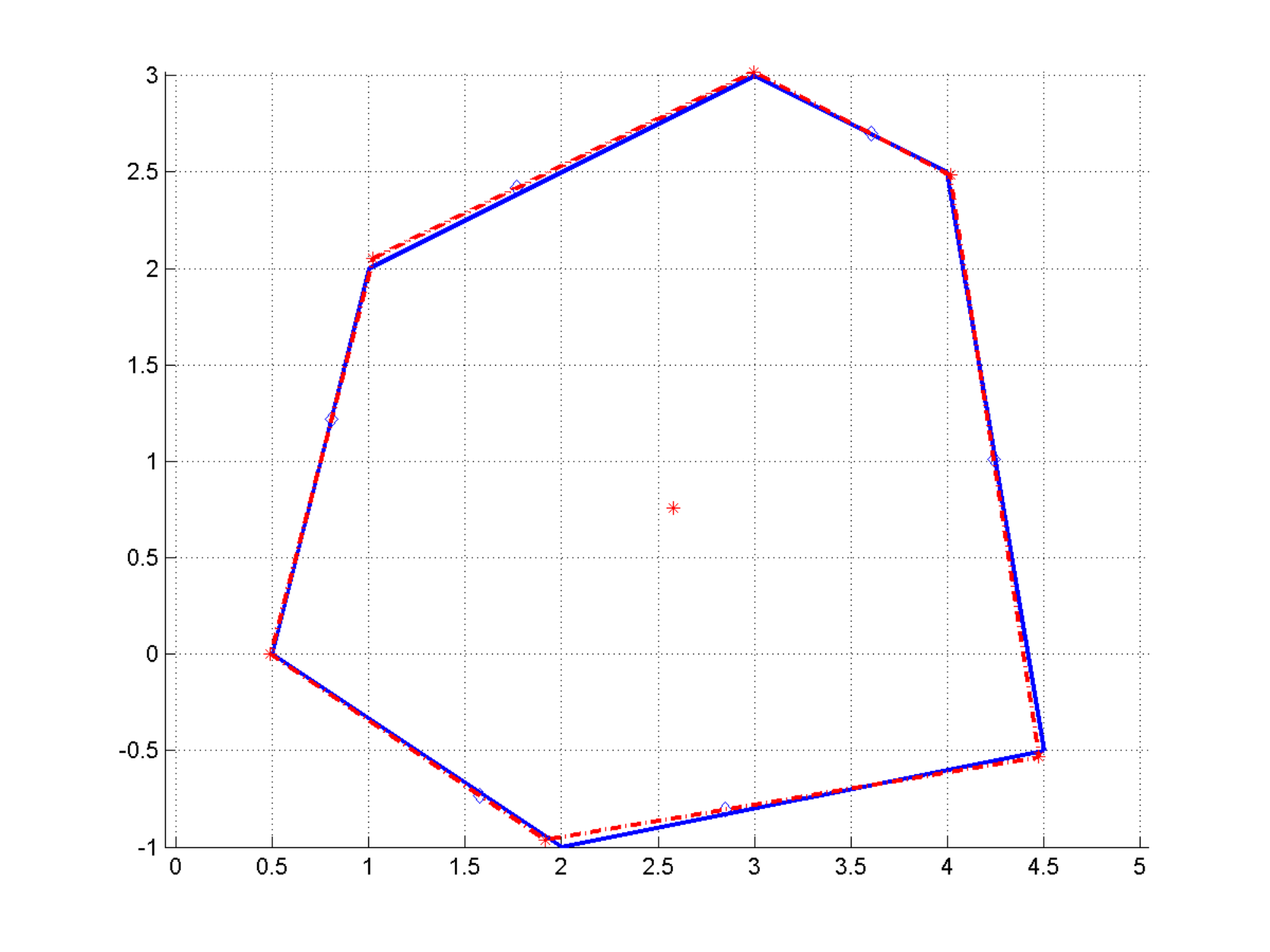}\hfill{}

\hfill{}(c)\hfill{}\hfill{}(d)\hfill{}

\caption{\label{fig:ex2:SS:reconstruction} Reconstruction of the sound-soft hexagonal obstacle with
(a) $k=6\pi$ with $5\%$ noise, (b) $k=6\pi$ without noise,
(c) $k=10\pi$ with $5\%$ noise, and (d) $k=10\pi$ without noise.  }
\end{figure}

\section{Concluding remarks}

In this paper, we develop an inverse scattering scheme of recovering a sound-hard or sound-soft polyhedral obstacle by only a few backscattering far-field measurements. It has been a very challenging issue in the literature on recovering an obstacle by minimum measurement data. We believe that the results in this work make some important contribution to this challenging issue. The proposed scheme proceeds with two steps. First, one uses the local maximum behavior of the modulus of the backscattering far-field data to determine the exterior normal direction of each of the side/face of the obstacle. Then, one can solve a small-scale finite dimensional algebraic problem to completely recover the obstacle. In order to justify the local maximum behavior of the modulus of the far-field data, we made essential use of the high-frequency asymptotics of the acoustic scattering. Our method can be extended to recovering non-convex obstacles, as well as to the electromagnetic scattering problems, which we shall report in the forthcoming work.

\section*{Acknowledgments}
This work was supported by the NSF of China under the grant No.\, 11201453 and No.\,11371115, and the FRG and startup funds from the Hong Kong Baptist University.


\clearpage

\begin{thebibliography}{99}


\bibitem{Ammari4} {H.~Ammari and H.~Kang}, {\it Reconstruction of Small Inhomogeneities from Boundary Measurements}, Lecture Notes in Mathematics, 1846. Springer-Verlag, Berlin Heidelberg, 2004.

\bibitem{AR} {G. Alessandrini and L. Rondi}, {\it Determining a sound-soft polyhedral scatterer by a single far-field measurement}, Proc. Amer. Math. Soc., {\bf 35} (2005), 1685--1691. Corrigendum: Preprtint arXiv math.Ap/0601406, 2006.

\bibitem{CWL} {S. N. Chandler-Wilde and S. Langdon}, {\it Acoustic scattering: high frequency boundary element methods and unified transform methods}, arxiv:1410.6137

\bibitem{CK} {D. Colton and R. Kress}, {\it Inverse Acoustic and Electromagnetic Scattering Theory}, 2nd Edition, Springer-Verlag, Berlin, 1998.

\bibitem{CS} {D. Colton and B. D. Sleeman}, {\it Uniqueness theorems for the inverse problem of acoustic scattering}, IMA J. Appl. Math., {\bf 31} (1983), 253--259.

\bibitem{CSV} {A.~R.~Conn, K. Scheinberg and L. N. Vicente}, {\it Introduction to Derivative-Free Optimization}, SIAM, Philadelphia, PA, 2009.

\bibitem{ElsYam} {J. Elschner and M. Yamamoto}, {\it Uniqueness in determining polygonal sound-hard obstacles with a single incoming wave}, Inverse Problems, {\bf 22} (2006), 355.

\bibitem{HLM} {D. P. Hewett, S. Langdon, and J. M. Melenk}, {\it A high frequency
hp boundary element method for scattering by convex polygons}, SIAM
J. Numer. Anal., {\bf 51} (2013), 629--653.

\bibitem{HonNakSin} {N. Honda, G. Nakamura and M. Sini}, {\it Analytic extension and reconstruction of obstacles from few measurements for elliptic second order operators}, Math. Ann., {\bf 355} (2013), 401--427.


\bibitem{Isa2} {V. Isakov}, {\it Inverse Problems for Partial Differential Equations}, 2nd edition, Applied Mathematical Sciences, 127, Springer-Verlag, New York, 2006.


\bibitem{LaxPhi} {P. D. Lax and R. S. Phillips}, {\it Scattering Theory}, Academic Press, 1967.


\bibitem{LLSS}
{ J. Li,  H. Liu,  Z. Shang and H. Sun},
{\it Two single-shot methods for locating multiple electromagnetic scatterers},
SIAM J. Appl. Math., {\bf 73} (2013), 1721--1746.


\bibitem{LLZ}
{J. Li, H. Liu and J. Zou},
{\it Locating multiple multiscale acoustic scatterers},
SIAM Multiscale Model. Simul., \textbf{12} (2014), 927--952.

\bibitem{Liu1} {H. Liu and J. Zou}, {\it Uniqueness in an inverse acoustic obstacle scattering problem for both sound-hard and sound-soft polyhedral scatterers}, Inverse Problems, {\bf 22} (2006), 515--524.


\bibitem{Kli} {M. Klibanov}, {\it Phaseless inverse scattering problems in three dimensions}, SIAM J. Appl. Math., {\bf 74} (2014), 392--410.

\bibitem{Maj} {A. Majda}, {\it High frequency asymptotics for the scattering matrix and the inverse problem
of acoustical scattering}. Comm. Pure Appl. Math., {\bf 29} (1976), 261--291.

\bibitem{McL} {W. McLean}, {\it Strongly Elliptic Systems and Boundary Integral Equations}, Cambridge University Press, 2000.

\bibitem{MelTay} {R. B. Melrose and M. E. Taylor}, {\it Near peak scattering and the corrected Kirchhoff
approximation for a convex obstacle}, Adv. Math., {\bf 55} (1985), 242--315.



\end{thebibliography}
\end{document}